\documentclass[11pt,a4paper,reqno]{amsart}
\usepackage[T1]{fontenc}
\usepackage{amsmath}
\usepackage{amsthm}
\usepackage{amsfonts}
\usepackage{amssymb}
\usepackage[pdftex]{graphicx,color,epsfig}
\usepackage{amsbsy}
\usepackage{mathrsfs}
\usepackage{bbm}
\usepackage{titletoc}
\newtheorem{theorem}{Theorem}[section]

\newtheorem{lemma}[theorem]{Lemma}

\newtheorem{proposition}[theorem]{Proposition}

\theoremstyle{remark}
\newtheorem{remark}[theorem]{\it \bf{Remark}\/}
\newenvironment{acknowledgement}{\noindent{\bf Acknowledgement.~}}{}
\numberwithin{equation}{section}
\catcode`@=11
\def\section{\@startsection{section}{1}%
  \z@{1.5\linespacing\@plus\linespacing}{.5\linespacing}%
  {\normalfont\bfseries\large\centering}}
\catcode`@=12
\newcommand{\be}{\begin{equation}}
\newcommand{\ee}{\end{equation}}
\newcommand{\bea}{\begin{eqnarray}}
\newcommand{\eea}{\end{eqnarray}}
\newcommand{\bee}{\begin{eqnarray*}}
\newcommand{\eee}{\end{eqnarray*}}

\def\pa{\partial}
\def\na{\nabla}
\def\Eq{\mbox{\rm Eq}}

\def\RR{\mathbb{R}}

\def\SS{\mathbb{S}}
\def\ds{\displaystyle}
\def\ni{\noindent}
\def\bs{\bigskip}
\def\ms{\medskip}

\def\eps{\varepsilon}
\def\fref#1{{\rm (\ref{#1})}}
\def\pref#1{{\rm \ref{#1}}}

\def\calH{{\mathcal H}}

\def\calC{{\mathcal C}}

\def\calE{{\mathcal E}}
\def\calO{{\mathcal O}}

\def\calX{{\mathcal X}}

\def\X{{\mathcal X}}%
\def\e{\varepsilon}
\catcode`@=11
\def\supess{\mathop{\operator@font Sup\,ess}}
\catcode`@=12
\def\un{{\mathbbmss{1}}}
\title[Orbital stability of spherical galactic models]{Orbital stability of spherical galactic models}
\author[M. Lemou]{Mohammed Lemou}
\address{CNRS and IRMAR, Universit\'e de Rennes 1, France}
\email{mohammed.lemou@univ-rennes1.fr}
\author[F. M\'ehats]{Florian M\'ehats}
\address{IRMAR, Universit\'e de Rennes 1, France}
\email{florian.mehats@univ-rennes1.fr}
\author[P. Rapha\"el]{Pierre Rapha\"el}
\address{IMT, Universit\'e Paul Sabatier, Toulouse, France}
\email{pierre.raphael@math.univ-toulouse.fr}
\dedicatory{Dedicated to the memory of our friend Naoufel Ben Abdallah}
\begin{document}

\begin{abstract}
 We consider the three dimensional gravitational Vlasov Poisson system which is a canonical model in astrophysics to describe the dynamics of galactic clusters. A well known conjecture \cite{binney} is the stability of spherical models which are nonincreasing radially symmetric steady states solutions. This conjecture was proved at the linear level by several authors in the continuation of the breakthrough work by Antonov \cite{A1} in 1961. In the previous work \cite{LMR8}, we derived the stability of anisotropic models under {\it spherically symmetric perturbations} using fundamental monotonicity properties of the Hamiltonian under suitable generalized symmetric rearrangements first observed in the physics litterature  \cite{lynden, gardner, WZS, aly}. In this work, we show how this approach combined with a {\it new generalized} Antonov type coercivity property implies the orbital stability of spherical models under general perturbations.
\end{abstract}


\maketitle
\titlecontents{section} 
[1.5em] 
{\vspace*{0.1em}\bf} 
{\contentslabel{2.3em}} 
{\hspace*{-2.3em}} 
{\titlerule*[0.5pc]{\rm.}\contentspage\vspace*{0.1em}}

\vspace*{-2mm}



\section{Introduction and main results}



\subsection{The gravitational Vlasov Poisson system} 


We consider  the three dimensional gravitational Vlasov-Poisson system
\be
\label{vp}
\left   \{ \begin{array}{llll}
         \ds \pa_t f+v\cdot\nabla_x f-\nabla
         \phi_f \cdot\nabla_v f=0, \qquad (t,x,v)\in \RR_+\times \RR^3\times \RR^3\\[3mm]
          \ds f(t=0,x,v)=f_0(x,v)\geq 0,\\[3mm]
         \end{array}
\right.
\ee
where, throughout this paper,
\be
\label{phif}
 \rho_f(x)=\int_{\RR^3} f(x,v)\,dv\quad \mbox{and}\quad  \phi_f(x)=-\frac{1}{4\pi |x|}\ast \rho_f
\ee
are the density and the gravitational Poisson field associated to $f$. This nonlinear transport equation is a well known model in astrophysics for the description of the mechanical state of a stellar system subject to its own gravity and the dynamics of galaxies, see for instance \cite{binney,fridman}.

The global Cauchy problem is solved in \cite{LP,Pf,S1} where unique global classical solutions $f(t)$ in $ \mathcal C^1_c$, the space of $\mathcal C^1$ compactly supported functions, are derived. Two fundamental properties of the nonlinear transport flow (\ref{vp}) are then first the preservation of the total Hamiltonian 
\be
\label{defhamiltonian}
\calH(f(t))=  
          \ds \frac{1}{2}\int_{\RR^6}
         |v|^2  f(t,x,v)dxdv-\frac{1}{2}\int_{\RR^3}|\nabla
          \phi_f(t,x)|^2dx=\calH (f(0)),
\ee
and second the preservation of all the so-called Casimir functions: $\forall G \in \calC^1([0,+\infty), \RR^+)$ such that $G(0)=0$, 
\be
\label{loi1}
\qquad \int_{\RR^6}G(f(t,x,v))\,dxdv=\int_{\RR^6}G(f_0(x,v))\,dxdv\,.
\ee
Equivalently, consider the distribution function associated to $f$:  \be
\label{muf}
\forall s\geq 0,\quad \mu_f(s)=\mbox{meas}\left\{(x,v)\in \RR^6:\,f(x,v)>s\right\},
\ee
then \fref{loi1} means the conservation law associated to nonlinear transportation: 
\be
\label{consevrationmeaurre}
\forall t\geq 0, \ \ \mu_f(t)=\mu_{f_0}.
\ee

In this paper, we will deal with weak solutions in the natural energy space 
\be
\label{calE}
\calE=\left\{f\geq 0 \mbox{ with }f\in L^1\cap L^\infty(\RR^6)\mbox{ and }|v|^2f\in L^1(\RR^6)\right\}.
\ee
For all $f_0\in \calE$, \fref{vp} admits a weak solution $f(t)$, constructed for instance in \cite{arsenev,HorstHunze,illner-neunzert}, which is also a renormalized solution, see \cite{dpl1,dpl2}. Moreover, this solution still satisfies \fref{loi1}, belongs to $\calC([0,+\infty),L^1(\RR^6))$ and the energy conservation \fref{defhamiltonian} is replaced by an inequality:
\be
\label{decayenergy}
\forall t\geq 0,\quad \calH(f(t))\leq \calH(f_0).
\ee


\subsection{Previous results} 


Jean's theorem \cite{Batt} gives a complete classification of radially symmetric steady state solutions to \fref{vp}. Recall that radial symmetry in our setting means  $f(x,v)\equiv f(|x|,|v|,x\cdot v).$ They are of the form $$Q(x,v)=F(e,\ell)$$ where $e,\ell$ are respectively the microscopic energy and the kinetic momentum 
\be
\label{def-el}
e(x,v)=\frac{|v|^2}{2}+\phi_Q(x), \ \ \ell=|x\wedge v|^2
\ee and are the {\it only} two invariants of the radially symmetric characteristic flow associated to the transport operator $\tau=v\cdot\nabla_x-\nabla\phi_Q\cdot\nabla_v$.

 A canonical problem which has attracted a considerable amount of works both in the physical and the mathematical community is the question of the {\it nonlinear stability of steady states models}. The {\it linear} stability of all nonincreasing anisotropic models satisfying
 \be
 \label{nonincresaing}
 \frac{\partial F}{\partial e}<0
 \ee
 is derived by Doremus, Baumann and Feix \cite{doremus} (see also \cite{gillon,kandrup1,SDLP} for related works), following the pioneering work by Antonov in the 60's \cite{A1,A2}. This analysis is based  on some coercivity properties of the linearized Hamiltonian under constraints formally arising from the linearization of the Casimir conservation laws (\ref{loi1}), see Lynden-Bell \cite{lynden}, known as Antonov's coercivity property.
 
 At the nonlinear level, the full orbital stability in the natural energy space $\calE$ has been obtained for specific subclasses of steady states as a direct consequence of Lions' concentration compactness principle \cite{PLL1,PLL2}, see \cite{wol, Guo,GuoRein2,GuoRein,Guo2,Dol, S2, LMR-note,LMR1,LMR5, SS}. This powerful strategy however only applies to specific models which are {\it global} minimizers of the Hamiltonian \fref{defhamiltonian} under at most two Casimir type conservation laws, see \cite{LMR1,LMR5} for a more complete introduction.

A first attempt to treat the general case and use the full rigidity provided by the {\it continuum} of conservation laws \fref{loi1} is proposed in \cite{GR4}, \cite{GuoLin} where the first result of stability against radially symmetric perturbations is obtained for the King model $F(e)= \left(\exp(e_{0}-e) -1\right)_{+}$. The approach is based on Antonov's coercivity property and a direct linearization of the Hamiltonian near the King profile.

We proposed in \cite{LMR8} a different approach based on fine monotonicity properties of the Hamiltonian under suitable generalized symmetric rearrangements as first observed in pioneering  breakthrough works in the physics litterature, see in particular Lynden-Bell \cite{lynden}, Gardner \cite{gardner}, Wiechen, Ziegler, Schindler \cite{WZS}, Aly \cite{aly}. This approach avoids the delicate step of linearization of the Hamiltonian and reduces the stability problem for the full distribution function $f$ to {\it a minimization problem for a generalized energy involving the Poisson field $\phi_f$ only}. The main outcome is the {\it radial} stability of nondecreasing anisotropic models, proved in \cite{LMR8}:

\begin{theorem}[Radial stability of nonincreasing anisotropic models, \cite{LMR8}]
\label{theoremradial}
\sloppy
Let $Q(x,v)=F(e,\ell)$ be a continuous, nonnegative compactly supported steady state solution to \fref{vp}. Assume that $Q$ is nonincreasing in the following sense: there exists $e_0<0$ such that $F$  is  $\calC^1$  on  $\calO=\{(e,\ell)\in \RR\times\RR_+\,:\,\, F(e,\ell)>0\}\subset (-\infty,e_0)\times \RR_+$ and $$\frac{\partial F}{\partial e}<0 \ \ \mbox{on} \ \ \mathcal O.$$ Then $Q$ is stable in the energy norm by radially symmetric perturbations, ie: for all $M>0$, for all $\eps>0$, there exists $\eta>0$ such that given $f_0\in \calC^1_c$ radially symmetric with
\be
\label{etavieux}
\|f_0-Q\|_{L^1}\leq \eta,\quad \|f_0\|_{L^\infty}\leq \|Q\|_{L^\infty} +M,\quad |\calH(f_0)-\calH(Q)|\leq \eta,
\ee
the corresponding global strong solution $f(t)$ to \fref{vp} satisfies: 
\be
\label{eta2vieux}
\forall t\geq 0, \ \ \|(1+|v|^2)(f(t)-Q)\|_{L^1}\leq \eps.
\ee
\end{theorem}


\subsection{Statement of the result}  


Our aim in this paper is to extend the stability result of Theorem \ref{theoremradial} to the full set of non radial perturbations. Here we recall that the radial problem enjoys an additional rigidity because for $f(x.v)$ radially symmetric, the Casimir conservation laws \fref{loi1} can be extended as follows: $\forall G(h,\ell)\geq 0$, $\mathcal C^1$ with $G(0,\ell)=0$, 
\be
\label{heoheoief}
\int_{\RR^6} G(f(t,x,v),|x\wedge v|^2)dxdv=\int_{\RR^6} G(f_0(x,v),|x\wedge v|^2)dxdv.
\ee
 This additional conservation law is fundamental in the proof of Theorem \ref{theoremradial}, and at the linear level, it is intimately connected to Antonov's coercivity property which is essentially equivalent to the coercivity of the Hessian of the Hamiltonian \fref{defhamiltonian} under the {\it full set of linearized constraints generated by \fref{heoheoief}}.
 
 For the full non radial problem, \fref{heoheoief} is lost. However, we claim that the strategy developped in \cite{LMR8} coupled with a {\it new generalized Antonov coercivity property} allows us to derive the classical conjecture of orbital stability of nonincreasing spherical models.
 
 \begin{theorem}[Orbital stability of spherical models]
 \label{thm}
 Let $Q$ be a continuous, nonnegative, non zero, compactly supported steady solution to \fref{vp}. Assume that $Q$ is a nonincreasing spherical model in the following sense: there exists a continuous function $F:\RR \to \RR_+$ such that 
\be
\label{Q}
\forall (x,v)\in \RR^6, \ \ Q(x,v)=F\left(\frac{|v|^2}{2}+\phi_Q(x)\right),
\ee
and there exists $e_0<0$ such that $F(e)=0$ for $e\geq e_0$, $F$ is $\calC^1$ on $(-\infty,e_0)$ and 
\be
\label{decroissance}
F'<0 \ \ \mbox{on} \ \ (-\infty,e_0).
\ee
Then $Q$ is orbitally stable in the energy norm by the flow \fref{vp}: for all $M>0$, for all $\eps>0$, there exists $\eta>0$ such that, given $f_0\in \calE$ with
\be
\label{eta}
\|f_0-Q\|_{L^1}\leq \eta,\quad \calH(f_0)\leq \calH(Q)+ \eta,\quad  \|f_0\|_{L^\infty}\leq  \|Q\|_{L^\infty}+M,
\ee
for any weak solution $f(t)$ to \fref{vp}, there exists a translation shift $z(t)$ such that $\forall t\geq 0$,
\be
\label{eta2}
  \| (1+|v|^2)(f(t,x,v)-Q(x-z(t),v))\|_{L^1(\RR^6)}\leq \eps.
\ee
\end{theorem}

\bs
\ni
{\it Comments on Theorem \pref{thm}.}

\ms
\ni
{\it 1. On the assumption on $Q$}. Jean's theorem \cite{Batt} ensures that the assumptions we make on $Q$ are very general. Note that  we allow $F'$ to blow up on the boundary $e\to e_0$ which is known to happen for many standard models. We in particular extract from \cite{binney} two models of physical relevance which fit into our analysis:
\begin{itemize}
\item[--] The generalized polytropic models:
$$F(e)=  \sum_{0\leq i,j\leq N} \alpha_{ij} (e_{0}-e)_{+}^{q_i} , \ \ 0<q_i<\frac{7}{2}, \ \ \alpha_{ij}\geq 0.$$
\item[--] The King model:
$$F(e)= \left(\exp(e_{0}-e) -1\right)_{+} \ \ \mbox{for some} \ \  e_0<0.$$
\end{itemize}

\ms
\ni
{\it 2. Anisotropic models}. Note that Theorem \ref{thm} deals with spherical models $Q=F(e)$ while the full class of anisotropic models $Q=F(e,\ell)$ is considered in Theorem \ref{theoremradial}. Let us insist that the orbital stability of all anisotropic models with respect to non radial perturbations {\it is not expected to hold in general} (see \cite{binney}) and nonradial instability mechanisms may happen induced by the non trivial dependence on kinetic momentum. We present a full non radial approach for spherical models only which is a canonical class, but which is likely not to be optimal. The derivation of sharp criterions of stability or instability for anisotropic models under non radial perturbations remains to be done.

\ms
\ni
{\it 3. Quantitative bounds}. The proof of Theorem \ref{thm} will rely on a compactness argument, and one could ask for more quantitative bounds. Such bounds are available for the Poisson field and a consequence of our analysis is that for $f\in \mathcal E$ satisfying \fref{eta}, we can find $z_f\in \RR^3$ such that 
$$\mathcal H(f)-\mathcal H(Q)+\|\phi_f\|_{L^\infty}\|f^*-Q^*\|_{L^1}\geq c_0\|\nabla \phi_f-\nabla \phi_Q(\cdot-z_f)\|_{L^2}^2$$ for some universal constant $c_0>0$, see \fref{controlH-fstar-J}, where $f^*$ and $Q^*$ denote respectively the usual symmetric decreasing rearrangements of $f$ and $Q$, as defined in Lemma \ref{schwarz}. The quantitative control of the full distribution function however seems to involve more subtle norms and would rely on weighted estimates for the bathtub principles, see \fref{identity2}. Such estimates were derived in the context of the incompressible 2D Euler in \cite{MarPul, MarWan}, but they seem to be more involved in our case due to the nonlinear structure of the generalized symmetric rearrangement that we consider, see \fref{dhieoyoye}.


\subsection{Strategy of the proof} Let us give a brief insight into the strategy of the proof of Theorem \ref{thm} which extends the approach introduced in \cite{LMR8}.


\ms
\ni
{\bf Step 1.} Monotonicity of the Hamiltonian under generalized symmetric rearrangements.

\ms
\ni
Let us define the Schwarz symmetrization of $f$ as 
\be
\label{shwartzsymm}
f^*(s)=\inf\{\tau\geq 0 \ \ : \ \ \mu_f(\tau)\leq s\},
\ee 
where $\mu_f$ is defined by \fref{muf}, which is the unique decreasing function on $\RR_+$ with $$\mu_{f}=\mu_{f^*}.$$ Given a potential $\phi$ in a suitable "Poisson field" class, we define the generalized symmetric nonincreasing rearrangement of $f$ with respect to the microscopic energy $e=\frac{|v|^2}{2}+\phi(x)$ as the unique function of $e$ which is equimeasurable to $f$, explicitely 
\be
\label{dhieoyoye}
f^{*\phi}(x,v)=f^*\circ a_{\phi}(e(x,v)), \ \ a_{\phi}(e)=\mbox{meas}\{(x,v)\in \RR^6, \ \ \frac{|v|^2}{2}+\phi(x)<e\}.
\ee
Any nonincreasing spherical steady state solution to \fref{vp} is a fixed point of this transformation when generated by its own Poisson field: 
\be
\label{vhohgoeri}
Q^{*\phi_Q}=Q.
\ee
Moreover, the Hamiltonian \fref{defhamiltonian} enjoys a nonlinear monotonicity property which was first observed in the physics litterature, see in particular Aly \cite{aly}: 
\be
\label{monotonicityhamiliotonanintro}
\mathcal H(f)\geq \mathcal H( f^{*\phi_f}).
\ee 
For perturbations which are equimeasurable to $Q$ ie 
\be
\label{hy[equi}
f^*=Q^*,
\ee
we can more precisely lower bound the Hamiltonian by a functional which depends on the Poisson field only: 
\be
\label{nvoehoehi}
\mathcal H(f)-\mathcal H(Q)\geq \mathcal J( \phi_f)-\mathcal J(\phi_Q)
\ee where $\mathcal J$ can be interpreted as a generalized energy, \cite{lynden}: $$\mathcal J(\phi_f)=\mathcal H(Q^{*\phi_f})+\frac12\int_{\RR^3}|\nabla \phi_{Q^{*\phi_f}}-\nabla \phi_{f}|^2.$$

\ms
\ni
{\bf Step 2.} Coercivity of the Hessian: a Poincar\'e inequality.

\ms
\ni
We now linearize the functional $\mathcal J$ at $\phi_Q$. The linear term drops thanks to the Euler-Lagrange equation \fref{vhohgoeri} and the Hessian takes the following remarkable form 
\be
\label{cneoeojo}
D^2\mathcal J(\phi_Q)(h,h)=\int_{\RR^3}|\nabla h|^2-\int_{\RR^6}|F'(e)|(h-\Pi h)^2dxdv
\ee
 where $\Pi$, defined by \fref{phiphi}, denotes after a suitable {\it phase space} change of variables the projection of $h$ onto the functions which depend only on the microscopic energy $e$. A similar structure occured in \cite{LMR8} where the corresponding quadratic form was
\be
\label{cnoeor}
\int_{\RR^3}|\nabla h|^2-\int_{\RR^6}\left|\frac{\partial F}{\partial e}(e,\ell)\right|(h-\Pi_{e,\ell} h)^2dxdv
\ee
and where $\Pi_{e,\ell}$ corresponds to the projection onto functions which depend on $(e,\ell)$ only ($e$ and $\ell$ being defined by \fref{def-el}). The strict coercivity of the quadratic form \fref{cnoeor} was then equivalent to Antonov's stability result, but this statement is no longer sufficient in our setting as \fref{cnoeor} is lower bounded by \fref{cneoeojo}.

We now claim the positivity of \fref{cneoeojo} for spherical models 
\be
\label{nvekoeoeu}
D^2\mathcal J(\phi_Q)(h,h)=\int_{\RR^3}|\nabla h|^2-\int_{\RR^6}|F'(e)|(h-\Pi h)^2dxdv\geq 0,
\ee
and in fact the quadratic form is coercive up to the degeneracy induced by translation invariance\footnote{see Proposition \ref{antonov} for precise statements}. For this, we reinterpret \fref{nvekoeoeu}  as a generalized Poincar\'e inequality with sharp constant, and we claim that the classical approach developed by H\"ormander \cite{horm1,horm2} for the proof of sharp weighted $L^2$ Poincar\'e inequalities: $$d\mu=e^{-V(x)}dx, \ \ \int_{\RR^N}(f-\overline{f})^2d\mu\lesssim  \int_{\RR^N}|\nabla f|^2d\mu, \ \ \overline{f}=\frac{\int_{\RR^N}fd\mu}{\int_{\RR^N}d\mu}$$ under the convexity assumption 
\be
\label{vnkoourgo}
\nabla^2V\gtrsim 1
\ee 
can be adapted to our setting. In particular, the non trivial convexity property \fref{vnkoourgo} appears in the setting of \fref{nvekoeoeu} as a consequence of the non linear structure of the steady state equation \fref{vhohgoeri}, see \fref{keyestimateg}.

\bs
\ni
{\bf Step 3.} Compactness up to translations.

\ms
\ni
The outcome of Step 2 is the variational characterization of $Q, \phi_Q$ respectively as the locally unique (up to translation shift) minimizers of the respectively constrained and unconstrained minimization problems $$\inf_{f^*=Q^*} \mathcal H(f),\ \ \inf \mathcal J(\phi).$$ More precisely, we will show that $\mathcal J(\phi)-\mathcal J(\phi_Q)$ controls the distance of $\phi$ to the manifold of translated Poisson fields $\phi_Q(\cdot+x)$, $x\in \RR^3$, see Proposition \ref{propJ}.

From standard continuity arguments,  the conservation law \fref{consevrationmeaurre} and the inequality \fref{decayenergy} ensure that Theorem \ref{thm} is now  equivalent to the relative compactness in the energy space up to translation of generalized minimizing sequences: $$f_n^*\to Q^* \ \ \mbox{in} \ \ L^1 \ \ \mbox{and} \ \ \limsup_{n\to +\infty}\mathcal H(f_n)\leq  \mathcal H(Q).$$ A slight improvement of the lower bound \fref{nvoehoehi} implies first the relative compactness up to translations $$\nabla \phi_{f_n}(\cdot+x_n)\to \nabla \phi_Q \ \ \mbox{in} \ \ L^2(\RR^3).$$ The strong convergence in the energy norm of the full distribution function  now follows from a further use of the extra terms in the monotonicity property \fref{monotonicityhamiliotonanintro} which yields: $$\int(1+|v|^2)|f_n(x+x_n,v)-Q(x,v)|dxdv \to 0 \ \ \mbox{as} \ \ n\to +\infty$$ and enables to conclude the proof of Theorem  \ref{thm}.
 
\medskip

The paper is organized as follows. In section 2, we show how a suitable phase-space symmetrization allows to reduce the study of the Hamiltonian $\mathcal H$ to the study of a  functional $\mathcal J$ which  depends on the Poisson field $\phi_f$ only. In section \ref{sectionJ},  we show that $\phi_{Q}$ is a local minimizer of this new functional  and that $\mathcal J(\phi)-\mathcal J(\phi_Q)$ controls the distance of $\phi$ to the manifold of translated functions $\phi_Q(\cdot+z)$, $z\in \RR^3$, Proposition \ref{propJ}.
In section 4, a sharp use of the monotonicity properties for both functionals $\mathcal H$ and $\mathcal J$  yields the compactness of the whole minimizing distribution functions. The proof of Theorem \ref{thm} then follows in section 5.

\bigskip

\begin{acknowledgement}
The authors are endebted to thank Frank Barthe who pointed out to us H\"ormander's proof of weighted Poincar\'e inequalities. M. Lemou was supported by the Agence Nationale de la Recherche, ANR Jeunes Chercheurs MNEC. F. M\'ehats was supported by the Agence Nationale de la Recherche, ANR project QUATRAIN. P. Rapha\"el was supported by the Agence Nationale de la Recherche, ANR projet SWAP. M. Lemou and F. M\'ehats acknowledge partial support from the ANR project CBDif.
\end{acknowledgement}


\section{Reduction to a functional of the gravitational potential}


In this section, we introduce the notion of rearrangement with respect to a given Poisson type field, and show the monotonicity of the Hamiltonian under the corresponding transformation which allows to compare the minimization problem of $\mathcal H(f)$ under the constraint $f^*=Q^*$ to an unconstrained minimization problem on the Poisson field $\phi_f$ only. Our approach extends the one we developed in \cite{LMR8} to the case of non radial potentials, and most arguments are in fact simplified by the absence of kinetic momentum.

\subsection{Properties of Poisson fields}


Let us start with defining a suitable class of "Poisson type" potentials:
$$
\mathbf \calX=\left\{\phi \in \calC(\RR^3)\mbox{ such that }\phi\leq 0,\,\,\lim_{|x|\to +\infty}\phi(x)=0,\,\,  \na \phi\in L^2(\RR^3) \mbox{ and  } m(\phi)>0\right\}
$$
where
 \be
\label{mphi}
m(\phi):=\inf_{x\in\RR^3}(1+|x|)|\phi(x)|
\ee
Notice that \fref{mphi} implies:
\be
\label{mphisomme}
\forall \phi,\tilde{\phi}\in \calX, \ \ \forall \lambda>0, \ \ m(\phi+\widetilde \phi)\geq m(\phi)+m(\widetilde \phi),\quad m(\lambda\phi)=\lambda m(\phi),
\ee
and thus $\calX$ is convex. Moreover, there holds:

\begin{lemma}[Properties of Poisson fields]
\label{lemX}
Let $f\in\calE$ nonzero and $\phi_f$ be its Poisson field given by \fref{phif}, then  $\phi_f\in \calX$.
\end{lemma}

\begin{proof}
Let $f\in \calE$, nonzero. From standard interpolation estimates, $\rho_f\in L^{5/3}\cap L^1$. Hence, by elliptic regularity, $\phi_f\in W^{2,5/3}_{loc}$, $\na \phi_f \in L^2(\RR^3)$ and $\phi_f\in C^{0,\frac{1}{5}}$ by Sobolev embedding. Also $\phi_f\leq 0$ and $\phi_f(x)\to 0$ as $|x|\to +\infty$ from \fref{phif}. In particular, $\phi_f$ attains its infimum on $\RR^3$ with $$-\infty<\min \phi_f\leq 0.$$It remains to show that $m(\phi)>0$ which follows from the existence of $C_{f}>0$ such that:
\be
\label{cphi}
\forall x\in \RR^3,\quad \phi(x)\leq -\frac{C_f}{1+|x|}.
\ee
Indeed, pick $R>0$ such that
$$\int_{|y|<R}\rho_f(y)dy\geq \frac{\|f\|_{L^1}}{2}>0,$$
and estimate for $|x|>R$:
$$-\phi(x)=\int_{\RR^3}\frac{\rho_f(y)}{4\pi|x-y|}dy\geq \int_{|y|<R}\frac{\rho_f(y)}{4\pi(|x|+R)}dy\geq \frac{\|f\|_{L^1}}{8\pi(|x|+R)},$$ which yields \fref{cphi}. The proof of Lemma \ref{lemX} is complete.
\end{proof}

Let us now associate to $\phi\in \calX$ the following Jacobian function:

\begin{lemma}[Properties of the Jacobian $a_{\phi}$]
\label{lemaphi}
Let $\phi\in \mathbf \calX$. We define the Jacobian function $a_\phi:\RR^*_-\to \RR^+$ as:
$$
\forall e<0, \ \ a_\phi(e)=\mbox{\rm meas}\left\{(x,v)\in \RR^6:\, \frac{|v|^2}{2}+\phi(x)<e\right\}.
$$
Then:
\begin{itemize}
\item[(i)] There holds the explicit formula:
\be
\label{aphi}
\forall e<0, \ \ a_\phi(e)=\frac{8\pi\sqrt{2}}{3} \int_{\RR^3}\left(e-\phi(x)\right)^{3/2}_+dx.
\ee
In particular, $a_\phi(e)=0$ for all $e<\min \phi$;
\item[(ii)] $a_\phi$ is $\calC^1$ on $(-\infty,0)$ and is a  strictly increasing $\calC^1$ diffeomorphism from $[\min \phi,0)$ onto $\RR_+$.
\end{itemize}
\end{lemma}
\begin{proof}
Let us prove (i). We have the inclusion
$$\left\{(x,v)\in \RR^6:\, \frac{|v|^2}{2}+\phi(x)<e\right\}\subset \left\{(x,v)\in \RR^6:\,\phi(x)<e\mbox{ and }|v|^2\leq 2(e-\min \phi)\right\}.$$
Let $e<0$. Since $\phi$ is continuous and goes to zero at the infinity, the set in the right-hand side is bounded in $\RR^6$, thus $a_\phi(e)<+\infty$. The formula \fref{aphi} now follows after passing to the spherical coordinates in velocity. We now prove (ii). Since, for all $e<0$, the set $\left\{x\in \RR^3:\,\phi(x)<e\right\}$
is bounded, we may apply the dominated convergence theorem and get the continuity and differentiability of $a_\phi$ on $\RR_-^*$, with
\be
\label{aphiprime}
a'_\phi(e)=4\pi\sqrt{2} \int_{\RR^3}\left(e-\phi(x)\right)^{1/2}_+dx.
\ee
Hence $a'_\phi$ is nonnegative and clearly continuous. Moreover, if $a'_\phi(e)=0$ then $e-\phi(x)\leq 0$ for all $x\in \RR^3$, which means that $e\leq \min \phi$. Therefore, if $e>\min \phi$, then $a'_\phi(e)>0$. It remains to prove that $\ds \lim_{e\to 0-}a_\phi(e)=+\infty$. Since $\phi\in \calX$, we have
$$a_\phi(e)\geq C\int_{\RR^3}\left(e+\frac{m(\phi)}{1+|x|}\right)_+^{3/2}dx\to +\infty \ \ \ \mbox{as} \ \ e\to 0,$$
from $\int_{\RR^3}\frac{dx}{(1+|x|)^{3/2}}=+\infty$, $m(\phi)>0$ and the monotone convergence theorem. This concludes the proof of Lemma \ref{lemaphi}.
\end{proof}


\subsection{Rearrangement with respect to the microscopic energy}


We introduce in this section the generalized rearrangement of $f$ with respect to a Poisson field $\phi\in \calX$. Let us start with recalling standard properties of the Schwarz symmetrization, \cite{kavian,lieb-loss,Mossino}. 
\begin{lemma}[Schwarz symmetrization or radial rearrangement]
\label{schwarz}
Let $f\in L^1_+\cap L^\infty$, then the Schwarz symmetrization $f^*$ of $f$ is the unique nonincreasing function on $\RR_+$ such that $f$ and $f^*$ have the same distribution function:
$$\forall s\geq 0,\quad \mu_f(s)=\mu_{f^*}(s)$$
with $\mu_f$ given by \fref{muf} and $\mu_{f^*}$ defined analogously\footnote{through the one dimensional Lebesgue measure}. Equivalently, $f^*$ is the pseudo-inverse of $\mu_f$:
$$\forall t\geq 0,\quad f^*(t)=\inf\left\{s\geq 0:\,\mu_f(s)\leq t\right\}.$$
The following properties hold:\\
(i) $f^*\in L^1_+\cap L^{\infty}$ with $$f^*(0)=\|f\|_{L^{\infty}}, \ \ {\rm Supp}(f^*)\subset[0,{\rm meas}({\rm Supp}(f))];$$ 
(ii) for all $\beta\in \calC^1(\RR_+,\RR_+)$ with $\beta(0)=0$,
\be
\label{characterization2}
\int_{\RR_+}\beta(f^*(t))dt=\int_{\RR^6}\beta(f(x,v))dxdv.
\ee
\end{lemma}
\ni
Observe that the above definition of $f^*$ is equivalent to
$$\forall t\geq 0,\quad f^*(t)=\sup\left\{s\geq 0:\,\mu_f(s)> t\right\},$$
with the convention that $f^*(t)=0$ when the set $\{s\geq 0:\,\mu_f(s)> t\}$ is empty. Note also that if $f$ is continuous then $f^*$ is continuous \cite{Talenti}. In particular, $Q^*$ is continuous.

Given $\phi\in \calX$, we now define the rearrangement of $f$ with respect to the microscopic energy $\frac{|v|^2}{2}+\phi(x)$ as follows:

\begin{lemma}[Symmetric rearrangement with respect to a microscopic energy]
\label{proprearrangement}
Let $f\in \calE$ and let $\phi\in \calX$. Let $f^*$ be the Schwarz rearrangement in $\RR^6$ given by Lemma \pref{schwarz}. We define the function
\be
\label{fstarphi}
f^{*\phi}(x,v)=
\left\{\begin{array}{ll}\ds f^*\left(a_\phi\left(\frac{|v|^2}{2}+\phi(x)\right)\right)&\ds\mbox{ if }\quad \frac{|v|^2}{2}+\phi(x)<0\\
0&\ds\mbox{ if }\quad \frac{|v|^2}{2}+\phi(x)\geq 0\end{array}\right.
\ee
on $\RR^6$, where $a_\phi$ is defined by \fref{aphi}. Then:
\begin{itemize}
\item[(i)] $f^{*\phi}$ is equimeasurable with $f$, i.e. 
\be
\label{eqf}
f^{*\phi}\in \Eq(f)=\{g\in L^1_+\cap L^{\infty} \ \ \mbox{with} \ \ \mu_f=\mu_g\}.
\ee
\item[(ii)] $f^{*\phi}$ belongs to the energy space, i.e. $f^{*\phi}\in \calE$ with
\be
\label{estimate}
\int_{\RR^6}\frac{|v|^2}{2}f^{*\phi}dxdv\leq C\|\na\phi\|_{L^2}^{4/3}\|f\|_{L^1}^{7/9}\|f\|_{L^\infty}^{2/9}.
\ee
\end{itemize}
\end{lemma}

\begin{proof}
Let us prove (i). The equimeasurability of $f$ and $f^{*\phi}$ relies on the following elementary change of variable formula: let two nonnegative function $\alpha\in \calC^0(\RR)\cap L^\infty(\RR)$ and $\gamma \in L^1(\RR_+)$, then
\bea
\int_{\frac{|v|^2}{2}+\phi(x)<0}\alpha\left(\frac{|v|^2}{2}+\phi(x)\right)\gamma\left(a_\phi\left(\frac{|v|^2}{2}+\phi(x)\right)\right)dxdv \qquad \qquad \qquad &&\nonumber\\
=\int_{\min \phi}^0 \alpha(e)\gamma(a_\phi(e))a_\phi'(e)de=\int_0^{+\infty} \alpha\left(a_\phi^{-1}(s)\right)\gamma(s)ds.&&\label{changeofvariable}
\eea
To obtain the first equality in \fref{changeofvariable}, we pass to the spherical coordinates in velocity $u=|v|$ and perform the change of variable $e=\frac{u^2}{2}+\phi(x)$ in the integral of $u$:
\bee
\int_{\frac{|v|^2}{2}+\phi(x)<0}\alpha\left(\frac{|v|^2}{2}+\phi(x)\right)\gamma\left(a_\phi\left(\frac{|v|^2}{2}+\phi(x)\right)\right)dxdv \qquad \qquad \qquad &&\\
=4\pi\sqrt{2}\int_{\RR^3}dx\int_{\phi(x)}^0 \alpha(e)\gamma(a_\phi(e))\left(e-\phi(x)\right)^{1/2}de.&&\\
=4\pi\sqrt{2}\int_{\min \phi}^0 \alpha(e)\gamma(a_\phi(e))de\int_{\RR^3}\left(e-\phi(x)\right)_+^{1/2}dx.&&
\eee
We conclude thanks to the formula \fref{aphiprime} of $a_\phi'$. The second equality comes after the change of variable $s=a_\phi(e)$. Recall from Lemma \ref{lemaphi} that $a_\phi$ is a $\calC^1$ diffeomorphism from $[\min \phi, 0)$ onto $\RR_+$.

Let $\beta\in \calC^1(\RR_+,\RR_+)$ such that $\beta(0)=0$. From \fref{changeofvariable} and the definition \fref{fstarphi}, we get
$$\int_{\RR^6}\beta\left(f^{*\phi}(x,v)\right)dxdv=\int_0^{+\infty}\beta(f^*(s))ds=\int_{\RR^6}\beta(f(x,v))dxdv,$$
where we use \fref{characterization2}. This proves that $f^{*\phi}\in \Eq(f)$.

Let us now prove (ii). From the equimeasurability of $f$ and $f^{*\phi}$, we already deduce that
\be
\label{equality}
\|f\|_{L^1}=\|f^{*\phi}\|_{L^1},\qquad \|f\|_{L^\infty}=\|f^{*\phi}\|_{L^\infty}.
\ee
Moreover, we have
\bee
\int_{\RR^6}\frac{|v|^2}{2}f^{*\phi}dxdv&=&\int_{\RR^6}\left(\frac{|v|^2}{2}+\phi(x)\right)f^{*\phi}dxdv-\int_{\RR^6}\phi(x)f^{*\phi}dxdv\\
&\leq &-\int_{\RR^6}\phi(x)f^{*\phi}dxdv\leq \|\phi\|_{L^{\infty}}\|f^*\|_{L^1}<+\infty,
\eee
where we used \fref{fstarphi}. More precisely:
\bee
\int_{\RR^6}\frac{|v|^2}{2}f^{*\phi}dxdv & \leq & -\int_{\RR^6}\phi(x)f^{*\phi}dxdv =\int_{\RR^3}\na\phi\cdot \na\phi_{f^{*\phi}}dx\\
&\leq &C\|\na\phi\|_{L^2}\||v|^2f^{*\phi}\|_{L^1}^{1/4}\|f^{*\phi}\|_{L^1}^{7/12}\|f^{*\phi}\|_{L^\infty}^{1/6}
\eee
where we used  the Cauchy-Schwarz inequality and the following standard interpolation inequality: for all $g\in \calE$,
\be
\label{interpolation}
\|\na\phi_g\|_{L^2}^2\leq C\||v|^2g\|_{L^1}^{1/2}\|g\|_{L^1}^{7/6}\|g\|_{L^\infty}^{1/3},
\ee
and \fref{estimate} follows. This concludes the proof of Lemma \ref{proprearrangement}.
\end{proof}

We end this subsection with an elementary lemma which will be useful in the sequel.

\begin{lemma}[Pseudo inverse of $f^*\circ a_{\phi}$]
\label{lemmapseudoinverse}
Let $f\in\calE$, nonzero, and $\phi\in \calX$. We define the pseudo inverse of $f^{*}\circ a_{\phi}$ for $s\in (0,\|f\|_{L^\infty})$ as:
\be
\label{defpseduoinverse}
(f^*\circ a_{\phi})^{-1}(s)=\sup\{e\in [\min \phi,0): \ \ f^*\circ a_{\phi}(e)>s\}.
\ee
Then $(f^*\circ a_{\phi})^{-1}$ is a nonincreasing function from $(0,\|f\|_{L^\infty})$ to $[\min \phi,0)$ and  for all $(x,v)\in \RR^6$ and $ s\in(0,\|f\|_{L^{\infty}})$,
\be
\label{implicationone}
 f^{*\phi}(x,v)>s \Longrightarrow \frac{|v|^2}{2}+\phi (x)\leq (f^*\circ a_{\phi})^{-1}(s),
\ee
\be
\label{implicationtwo}
  f^{*\phi}(x,v)\leq s \Longrightarrow \frac{|v|^2}{2}+\phi (x)\geq (f^*\circ a_{\phi})^{-1}(s),
 \ee
 where $f^{*\phi}$ is defined by \fref{fstarphi}.
\end{lemma}

\begin{proof} Let $s\in(0,\|f\|_{L^\infty})$, then from $f^*(0)=\|f\|_{L^{\infty}}$, $f^*(t)\to 0$ as $t\to +\infty$ and Lemma \ref{lemaphi},
\be
\label{cohoeh}
\{e\in [\min \phi,0): \ \ f^*\circ a_{\phi}(e)>s\} \ \ \mbox{is not empty}
\ee
and $(f^*\circ a_{\phi})^{-1}(s)$ defined by \fref{defpseduoinverse} is strictly negative.
The monotonicity of $(f^*\circ a_{\phi})^{-1}$ follows from the monotonicity of $f^*$ and $a_{\phi}$. Assume that $f^{*\phi}(x,v)> s$, then from the definition \fref{fstarphi}, we have $\min \phi\leq  \frac{|v|^2}{2}+\phi (x)<0$. We also have $f^*\circ a_{\phi}(\frac{|v|^2}{2}+\phi (x))> s$, therefore  $\frac{|v|^2}{2}+\phi (x)\leq (f^*\circ a_{\phi})^{-1}(s)$ from the definition (\ref{defpseduoinverse}). This proves \fref{implicationone}.
Assume now that $f^{*\phi}(x,v)\leq s$. Then, for all $e\in\{\tilde e\in [\min \phi,0): \ \ f^*\circ a_{\phi}(\tilde e)>s\}$ which is a non empty set, we have $\frac{|v|^2}{2}+\phi (x)> e$, and (\ref{implicationtwo}) is proved.
\end{proof}

\subsection{Spherical models are fixed points of the generalized rearrangement}

We now reinterpret the assumptions on $Q$ in Theorem \ref{thm} and claim that spherical models are fixed points of the $f\to f^{*\phi_{f}}$ transformation\footnote{Note that this is essentially a characterization of spherical models}.
\begin{lemma}[$Q$ is a fixed point of the $f^{*\phi_f}$ rearrangement]
\label{id-Q}
Let $Q$ be a radially symmetric spherical models as in the assumptions of Theorem \pref{thm}. Then we have
\be
\label{Qfixe}
 F(e)=Q^*\circ a_{\phi_{Q}}(e), \quad \forall e\in [\phi_{Q}(0),0), \qquad  \mbox{and} \qquad Q^{*\phi_{Q}}=Q \quad on \  \RR^6.
\ee
\end{lemma}
\begin{proof}
Observe first that, since the boundary of $\{Q(x,v)>0\}$ is the level set $\frac{|v|^2}{2}+\phi_Q(x)=e_0$, we have $\mu_Q(0)=\mbox{meas (Supp}(Q))$. From the equimeasurability of $Q$ and $Q^*$, we have 
\bee
\mu_Q(F(e))&=&\mbox{meas}\left\{(x,v)\in \RR^6,\  F\left(\frac{|v|^2}{2}+\phi_Q(x)\right) > F(e)\right\}\\
&=&  \mbox{meas}\{s\in \RR^*_{+}, \ Q^*(s) > F(e)\},
\eee
for all $e\leq e_0$.  Since $F$ is strictly decreasing on $(-\infty, e_{0}]$, this is equivalent to
\be
\label{aq1} \mu_{Q}(F(e))=a_{\phi_{Q}}(e) =  \mbox{meas}\left\{s\in \RR^*_{+}, \ Q^*(s) > F(e)\right\},  \qquad \forall e\leq  e_{0}.
\ee
In particular $a_{\phi_{Q}}(e_{0})=\mbox{meas}(\mbox{Supp}(Q)) >0$, which implies that $\phi_{Q}(0) < e_{0}$. 
From \fref{aq1} and the invertibility of both continuous functions $F$ and $a_{\phi_Q}$ on $[\phi_{Q}(0), e_{0}]$, we deduce that $\mu_{Q}$ is continuous and one-to-one from $[0,F(\phi_{Q}(0))]$ to $[0,a_{\phi_{Q}}(e_{0})]$. In particular, $Q^*$ is the inverse of $\mu_Q$ on this interval (and not only its pseudo-inverse) and we have
\be
\label{aq2} Q^*\circ a_{\phi_{Q}}(e) =  F(e), \qquad \forall e \in [\phi_{Q}(0), e_{0}].
\ee
Identity \fref{aq2} is still valid for $e_{0}<e<0$. Indeed, in this case, we have $F(e)=0$,  and
$a_{\phi_{Q}}(e) > a_{\phi_{Q}}(e_{0})= \mbox{meas}(\mbox{Supp}(Q)) $,  which implies that $Q^*\circ a_{\phi_{Q}}(e) =0$.
The first identity of \fref{Qfixe} is then proved.

Now, the identity $Q^{*\phi_{Q}}=Q$ is a straightforward consequence of the first identity of  \fref{Qfixe}. Indeed, we first observe that $\frac{|v|^2}{2}+\phi_Q(x) \geq \phi_{Q}(0)$. If $\frac{|v|^2}{2}+\phi_Q(x) \geq 0$ then
$F\left(\frac{|v|^2}{2}+\phi_Q(x) \right) =0$ and $Q^{*\phi_{Q}}(x,v)=0$ from the definitions of $F$ and $Q^{*\phi_{Q}}$. If $\frac{|v|^2}{2}+\phi_Q(x) < 0$, then we apply the first identity to $e=\frac{|v|^2}{2}+\phi_Q(x)$  and get the desired equality. The proof of Lemma \ref{id-Q} is complete.
\end{proof}


\subsection{Monotonicity of the Hamiltonian under symmetric rearrangement}


We are now in position to derive the monotonicity of the Hamiltonian under the generalized rearrangement which is the first key to our analysis and was already observed in the physics litterature, see \cite{aly} and references therein. Given $f\in\calE\setminus\{0\}$, by Lemma \ref{lemX} we have $\phi_f\in \X$ and we will note to ease notation:
\be
\label{deffhat}
\widehat{f}=f^{*\phi_f}.
\ee
Given $\phi\in \calX$, we define the functional 
\be
\label{defiphi}
 \mathcal J _{f^*}(\phi) =\calH(f^{*\phi})+\frac{1}{2}\|\na \phi-\na \phi_{f^{*\phi}}\|^2_{L^2}
\ee
which is well defined from Proposition \ref{proprearrangement}. We claim:

\begin{proposition}[Monotonicity of the Hamiltonian under the $f^{*\phi_f}$ rearrangement]
\label{propclemonotonie}
Let $f\in\calE\setminus \{0\}$ and $\widehat{f}$ given by \fref{deffhat}, then:
\be
\label{keymononicity}
\calH(f)\geq  \mathcal J _{f^*}(\phi_f)\geq \calH(\widehat{f}).
\ee
Moreover, $\calH(f)= \calH(\widehat{f})$ if and only if $f=\widehat{f}$.
\end{proposition}

\begin{proof} First compute for all $(f,g)\in \calE$:
\bea
\label{identity}
\nonumber \calH(f)& = & \frac{1}{2}\int_{\RR^6}|v|^2f-\frac{1}{2}\int_{\RR^3}|\nabla\phi_f|^2\\
\nonumber & = & \int_{\RR^6}\left(\frac{|v|^2}{2}+\phi_f\right)(f-g)+\frac{1}{2}\int_{\RR^6}|v|^2 g+\int_{\RR^3}\phi_fg+\frac{1}{2}\int|\nabla\phi_f|^2\\
& = & \calH(g)+\frac{1}{2}\|\na\phi_f-\na\phi_{g}\|_{L^2}^2+\int_{\RR^6}\left(\frac{|v|^2}{2}+\phi_f(x)\right)(f-g).
\eea
Replacing $g$ by $\widehat f=f^{*\phi_f}$ yields from \fref{defiphi}:
\be
\label{identitypartic}
\calH(f)=\mathcal J _{f^*}(\phi_f)+\int_{\RR^6}\left(\frac{|v|^2}{2}+\phi_f(x)\right)(f-f^{*\phi_f})\,dxdv,
\ee
and hence \fref{propclemonotonie} follows from 
\be
\label{identity2}
 \int_{\RR^6}\left(\frac{|v|^2}{2}+\phi_f(x)\right)(f-\widehat{f})\,dxdv\geq 0,
\ee
with equality if  and only if $f=\widehat{f}$. 
The proof of (\ref{identity2}) is reminiscent from the standard inequality for symmetric rearrangement, see \cite{lieb-loss}: $$\int_{\RR^6}|x|f^*\leq\int_{\RR^6}|x|f.$$ 
Indeed, use the layer cake representation $$f(x,v)=\int_{t=0}^{\|f\|_{L^{\infty}}}\un_{t<f(x,v)} dt$$ and Fubini to derive:
\bea
\label{mesureurto}
\nonumber & & \hspace*{-1cm}\int_{\RR^6}\left(\frac{|v|^2}{2}+\phi_f\right)(f-\widehat f)\,dxdv\\
\nonumber &&=   \int_ {t=0}^{\|f\|_{L^{\infty}}}dt\int_{\RR^6}\left(\un_{t<f(x,v)} -\un_{t<\widehat f(x,v)}\right)\left(\frac{|v|^2}{2}+\phi_f\right)dxdv\\ 
\nonumber && = \int_{t=0}^{\|f\|_{L^{\infty}}}dt\int_{\RR^6}\left(\un_{\widehat f(x,v)\leq t<f(x,v)}-\un_{f(x,v)\leq t<\widehat f(x,v)}\right)\left(\frac{|v|^2}{2}+\phi_f\right)dxdv\\
& & = \int_{t=0}^{\|f\|_{L^{\infty}}}dt\left(\int_{S_1(t)}\left(\frac{|v|^2}{2}+\phi_f\right)dxdv-\int_{S_2(t)}\left(\frac{|v|^2}{2}+\phi_f\right)dxdv\right)
\eea
with $$S_1(t)=\{\widehat f(x,v)\leq t<f(x,v)\},\qquad S_2(t)=\{f(x,v)\leq t<\widehat f(x,v)\}.$$
Observe from $\widehat{f}\in \Eq(f)$ that: 
\be
\label{identitymeasures}
\mbox{for a.e.} \ \ t>0, \qquad \mbox{meas}(S_1(t))=\mbox{meas}(S_2(t)).
\ee 
We thus conclude from \fref{implicationone} and (\ref{identitymeasures}): $\forall t\in (0,\|f\|_{L^{\infty}})$,
$$\int_{S_2(t)}\left(\frac{|v|^2}{2}+\phi_f(x)\right)dxdv\leq \mbox{meas}(S_2(t))(f^{*}\circ a_{\phi_f})^{-1}(t)=\int_{S_1(t)} (f^{*}\circ a_{\phi_f})^{-1}(t)dxdv.$$ 
Injecting this into (\ref{mesureurto}) together with (\ref{implicationtwo}) yields:
\bee 
&&\hspace*{-1cm}\int_{\RR^6}\left(\frac{|v|^2}{2}+\phi_f\right)(f-\widehat f)\,dxdv \geq \\
&&\qquad  \int_0^{\|f\|_{L^{\infty}}}dt \int_{S_1(t)}\left(\frac{|v|^2}{2}+\phi_f(x)-(f^{*}\circ a_{\phi_f})^{-1}(t)\right)dxdv\geq 0
\eee
and \fref{identity2} is proved. We also have the analogous inequality for $S_2(t)$:
\bee 
&&\hspace*{-1cm}\int_{\RR^6}\left(\frac{|v|^2}{2}+\phi_f\right)(f-\widehat f)\,dxdv \geq \\
&&\qquad  \int_0^{\|f\|_{L^{\infty}}}dt \int_{S_2(t)}\left((f^{*}\circ a_{\phi_f})^{-1}(t)-\frac{|v|^2}{2}-\phi_f(x)\right)dxdv\geq 0.
\eee
 Let us now study the case of equality in \fref{identity2}. If $$ \int_{\RR^6}\left(\frac{|v|^2}{2}+\phi_f(x)\right)(f-\widehat{f})\,dxdv=0,$$ 
the above chain of equalities implies that for a.e $t>0$, either $\mbox{meas}(S_1(t))=\mbox{meas}(S_2(t))=0$ or, a.e. $ (x_1,v_1)\in S_1(t)$,  a.e $(x_2,v_2)\in S_2(t)$,
$$\frac{|v_1|^2}{2}+\phi_f(x_1)=(f^{*}\circ a_{\phi_f})^{-1}(t)=\frac{|v_2|^2}{2}+\phi_f(x_2).$$
The last assertion contradicts the fact that $\widehat f(x_1,v_1)\leq t<\widehat f(x_2,v_2)$. Therefore, for a.e $t\in(0,\|f\|_{L^{\infty}})$, we have $\mbox{meas}(S_1(t))=\mbox{meas}(S_2(t))=0$. On the other hand, $\|f\|_{L^{\infty}}=\|f^*\|_{L^{\infty}}$ and hence $\mbox{meas}(S_1(t))=\mbox{meas}(S_2(t))=0$ for $t>\|f\|_{L^{\infty}}$. Hence $\mbox{meas}(S_1(t))=\mbox{meas}(S_2(t))=0$ for a.e. $t>0$, which implies $f=\widehat f$. This concludes the proof of Proposition \ref{propclemonotonie}.
\end{proof}


\section{Study of the reduced functional $\mathcal J$}
\label{sectionJ}

In this section, we focus onto the functional on $\calX$:
\be
\label{J}
 \mathcal J(\phi) =\mathcal J_{Q^*}(\phi)=\calH(Q^{*\phi})+\frac{1}{2}\|\na \phi-\na \phi_{Q^{*\phi}}\|^2_{L^2}
\ee
We claim that locally near $\phi_Q$, $\mathcal J(\phi)-\mathcal J(\phi_Q)$ is equivalent to the distance of $\phi$ to the manifold of translated Poisson fields $\phi_Q(\cdot+x)$, $x\in \RR^3$.

\begin{proposition}[Coercive behavior of $\mathcal J$ near $\phi_Q$]
\label{propJ}
There exist universal constants $c_0,\delta_0>0$ and a continuous map $\phi\to z_{\phi}$ from $(\dot{H}^1,\|\cdot\|_{\dot{H}^1})\rightarrow \RR^3$ such that the following holds true. Let $\phi\in \X$ with 
\be
\label{linfitysmallnes}
\inf_{z\in \RR^3}\left(\|\phi-\phi_Q(\cdot-z)\|_{L^{\infty}}+\|\nabla\phi-\nabla\phi_Q(\cdot-z)\|_{L^{2}}\right)<\delta_0,
\ee
then:
\be
\label{coercivityloc}
\mathcal J(\phi)-\mathcal J(\phi_Q)\geq c_0\|\nabla \phi-\nabla \phi_Q(\cdot-z_{\phi})\|_{L^2}^2.
\ee
\end{proposition}
This section will be devoted to the proof of  Proposition \ref{propJ} which relies first on the second order Taylor expansion of $\mathcal J$ at $\phi_Q$, Proposition \ref{lemA1}, and then on the coercivity of the Hessian which is the second main key to our analysis, Proposition \ref{antonov}, and corresponds to a generalized Antonov's coercivity property.


\subsection{Differentiability of $\mathcal J$}
\label{sectionjdiff}


Our aim in this section is to prove the differentiability of $\mathcal J$ at $\phi_Q$ and to compute the first two derivatives.

Let us start with differentiability properties of the function $\phi\mapsto a_{\phi}$ defined in Lemma \ref{lemaphi}, see Appendix \ref{appendixa} for the proof.
  
\begin{lemma}[Continuity and differentiability properties of $\phi\mapsto a_{\phi}$]
\label{lemdiffa}
Let $\phi,\widetilde \phi\in \calX$ and let $h=\phi-\widetilde \phi$. Then the following holds.\\
(i) The function $(\lambda,e)\mapsto a_{\phi+\lambda h}(e)$ is a $\calC^1$ function on $[0,1]\times \RR_-^*$ and we have
\be
\label{derivaphi}
\frac{\pa}{\pa \lambda}a_{\phi+\lambda h}(e)=-4\pi\sqrt{2}\int_{\RR^3}(e-\phi(x)-\lambda h(x))_+^{1/2}h(x)dx.
\ee
(ii) Let $s\in \RR_+^*$. Then the function $\lambda\mapsto a_{\phi+\lambda h}^{-1}(s)$ is differentiable on $[0,1]$ and we have
\be
\label{derivaphim1}
\frac{\pa}{\pa \lambda}a_{\phi+\lambda h}^{-1}(s)=
\frac{\ds \int_{\RR^3}(a^{-1}_{\phi+\lambda h}(s)-\phi(x)-\lambda h(x))_+^{1/2}h(x)dx}{{\ds \int_{\RR^3}(a^{-1}_{\phi+\lambda h}(s)-\phi(x)-\lambda h(x))_+^{1/2}dx}}\,.
\ee
\end{lemma}

We are now in position to differentiate the functional $\mathcal J$.

\begin{proposition}[Differentiability of $\mathcal J$]
\label{lemA1}
The functional $\mathcal J$ defined by \fref{J} on $\calX$ satisfies the following properties.\\
(i) {\em Differentiability of $\mathcal J$.} Let $\phi,\widetilde \phi \in \X$, then the function
$$\lambda\mapsto \mathcal J(\phi+\lambda(\widetilde \phi-\phi))$$ is twice differentiable on $[0,1]$.\\
(ii) {\em Taylor expansion of $\mathcal J$ near $\phi_Q$.}
There holds the Taylor expansion near $\phi_Q$: $\forall \phi\in \calX$,
\be
\label{taylor}
\mathcal J(\phi)-\mathcal J(\phi_Q)=\frac{1}{2}D^2\mathcal J(\phi_Q)(\phi-\phi_Q,\phi-\phi_Q)+\eta\left(\|\phi-\phi_Q\|_{L^\infty}\right)\|\na \phi-\na \phi_Q\|_{L^2}^2
\ee
where
$$\eta(\delta)\to 0 \mbox{ as }\delta \to 0.$$
Moreover, the second derivative of $\mathcal J$ at $\phi_Q$ in the direction $h$ is given by
\bea
\label{deriv2}
& & \hspace*{-5mm}D^2\mathcal J(\phi_Q)(h,h)\\
\nonumber & &= \int_{\RR^3}|\na h|^2\,dx-\int_{\RR^6}\left|F'\left(\frac{|v|^2}{2}+\phi_Q(x)\right)\right|\left(h(x)-\Pi h(x,v)\right)^2 dxdv\,,
\eea
where $\Pi h$ is the projector:
\be
\label{phiphi}
\Pi h(x,v)=\frac{\ds \int_{\RR^3}\left(\frac{|v|^2}{2}+\phi_Q(x)-\phi_Q(y)\right)_+^{1/2}h(y)dy}{\ds \int_{\RR^3}\left(\frac{|v|^2}{2}+\phi_Q(x)-\phi_Q(y)\right)_+^{1/2}dy}.
\ee
\end{proposition}

\begin{remark} The projector $\Pi h$ given by \fref{phiphi} should be understood as the projector onto the functions which depend only on the microscopic energy $e(x,v)=\frac{|v|^2}{2}+\phi_Q(x).$
\end{remark}

\begin{proof}
Let us decompose $\mathcal J$ into a kinetic and a potential part:
\be
\label{JJ}
\mathcal J (\phi)= \mathcal J _{Q^*}(\phi) =\calH(Q^{*\phi})+\frac{1}{2}\|\na \phi-\na \phi_{Q^{*\phi}}\|^2=\frac{1}{2}\int|\nabla \phi|^2dx+ \mathcal J_0(\phi)\ee
with 
\be
\label{A0}
\mathcal J_0(\phi)=\int_{\RR^6}\left(\frac{|v|^2}{2}+\phi(x)\right)Q^{*\phi}(x,v)\,dxdv.
\ee
Note from Proposition \ref{proprearrangement} that $Q^{*\phi}\in \calE$ and is supported in $\frac{|v|^2}{2}+\phi<0$, thus
$$-\infty<\mathcal J_0(\phi)\leq 0.$$
Let $\phi,\widetilde \phi \in \X$ and let $h=\widetilde \phi-\phi$. We shall differentiate with respect to $\lambda$ the function $\mathcal J_0(\phi+\lambda h)$.

\bs
\ni
{\em Step 1. First derivative of $\mathcal J_0$}.

\ms
\ni
Introduce the following primitive of $Q^*$:
\be
\label{G}
G(s)=\int_0^s Q^*(\sigma)d\sigma,
\ee
which is a uniformly bounded $\calC^1$ function with bounded derivative, since by assumption $Q$ (thus $Q^*$) is continuous and compactly supported. We first transform the expression \fref{A0} of $\mathcal J_0$. By making the change of variable in velocity $e=\frac{|v|^2}{2}+\phi$ and using \fref{aphiprime}, we get
\bee
\mathcal J_0(\phi) & = & \int_{\min \phi}^0 eQ^*\left(a_\phi(e)\right) a_\phi'(e)\,de=\int_{\min \phi}^0 e\left(G\circ a_\phi\right)'(e)\,de\\
& = &\left[eG(a_{\phi}(e))\right]_{\min \phi}^0-\int_{\min \phi}^0 G\left(a_\phi(e)\right)de=  -\int_{-\infty}^0 G\left(a_\phi(e)\right)de.
\eee
Note that the boundary term is dropped thanks to the definition \fref{G} and the following properties:
$$a_\phi(\min \phi)=0, \quad \lim_{e\to 0-}a_\phi(e)=+\infty,\quad \int_0^{+\infty}Q^*(\sigma)d\sigma=\|Q\|_{L^1}<+\infty.$$

In order to differentiate $ \mathcal J_0(\phi+\lambda h)$ with respect to $\lambda$, we now use \fref{derivaphi} and the $\calC^1$ smoothness of $G$ to derive: $\forall e<0$
$$\frac{\pa}{\pa\lambda}G(a_{\phi+\lambda h}(e))=-4\pi\sqrt{2}\,Q^*\left(a_{\phi+\lambda h}(e)\right)\int_{\RR^3}(e-\phi(x)-\lambda h(x))_+^{1/2}h(x)dx\,.
$$
Recall that we have $\mbox{Supp }(Q^*)=[0,L_0]$, with
$$L_0=\mbox{meas (Supp $Q$)}<+\infty.$$
Hence, from Lemma \ref{lemdiffa} {\em (i)}, we deduce that there exists $e_1<e_2<0$ such that
\be
\label{support}
\left\{(\lambda,e)\in [0,1]\times \RR_-^*:\, a_{\phi+\lambda h}(e)\in \mbox{Supp}(Q^*)\right\}\subset [0,1]\times [e_1,e_2].
\ee
Moreover, we have the following uniform bound: for all $(\lambda,e)$,
\bee
\left|\frac{\pa}{\pa\lambda}G(a_{\phi+\lambda h}(e))\right|&\leq&4\pi\sqrt{2}\|Q^*\|_{L^\infty}\int_{\RR^3}(e_2-(1-\lambda)\phi(x)-\lambda \widetilde\phi(x))_+^{1/2}h(x)dx\\
&\leq&4\pi\sqrt{2}\|Q^*\|_{L^\infty}\int_{\RR^3}(e_2-\phi(x)-\widetilde \phi(x))_+^{1/2}h(x)dx<+\infty.
\eee
Therefore, Lebesgue's derivation theorem ensures:
\bea
\label{lam1}
& & \hspace*{-1.5cm}\nonumber \frac{\pa}{\pa\lambda}\mathcal J_0(\phi+\lambda h) =  -\frac{\pa}{\pa\lambda}\int_{-\infty}^0 G\left(a_{\phi+\lambda h}(e)\right)de\\
 & =&   4\pi\sqrt{2}\int_{-\infty}^0\int_{\RR^3}Q^*\left(a_{\phi+\lambda h}(e)\right)(e-\phi(x)-\lambda h(x))_+^{1/2}h(x)\,dxde.
\eea

\bs
\ni
{\em Step 2. Second derivative of $\mathcal J_0$}.

\ms
\ni
Let us now compute the second derivative of $\mathcal J_0(\phi+\lambda h)$ with respect to $\lambda$. First, an integration by parts of \fref{lam1} with respect to the variable $e$ gives
$$
\frac{\pa}{\pa\lambda}\mathcal J_0(\phi+\lambda h)=-\frac{8\pi\sqrt{2}}{3}\int_{-\infty}^0\int_{\RR^3}{Q^*}'\left(a_{\phi+\lambda h}(e)\right)a_{\phi+\lambda h}'(e)(e-\phi(x)-\lambda h(x))_+^{3/2}h(x)\,dxde.
$$
Applying the change of variable $s=a_{\phi+\lambda h}(e)$, we obtain
\bea
\nonumber
\frac{\pa}{\pa\lambda}\mathcal J_0(\phi+\lambda h) &=&-\frac{8\pi\sqrt{2}}{3}\int_{0}^{L_0}ds\,{Q^*}'(s)\int_{\RR^3}(a_{\phi+\lambda h}^{-1}(s)-\phi(x)-\lambda h(x))_+^{3/2}h(x)dx\\
&=&-\frac{8\pi\sqrt{2}}{3}\int_{0}^{L_0}\int_{\RR^3}{Q^*}'(s)g(\lambda,x,s)h(x)dsdx,
\label{integ}
\eea
with
$$
g(\lambda,x,s)=(a_{\phi+\lambda h}^{-1}(s)-\phi(x)-\lambda h(x))_+^{3/2}.
$$
Recall that, by \fref{support}, the quantity $e=a_{\phi+\lambda h}^{-1}(s)$ can be restricted to some interval $[e_1,e_2]$ in this integral, with $e_1<e_2<0$. Moreover, as in Step 1 of the proof of Lemma \ref{lemdiffa}, one deduces from the decay at the infinity of $\phi$ and $\widetilde \phi$ that the domain
$$\left\{x\in\RR^3:\,\phi(x)+\lambda h(x)\right\}\leq e_2$$
is bounded independently of $\lambda$. Therefore, the variable $x$ in the integral \fref{integ} can be restricted to a bounded domain.

Let us differentiate \fref{integ} with respect to $\lambda$. From \fref{derivaphim1}, one gets
\bee
\frac{\pa}{\pa \lambda}g(\lambda,x,s)&=& -\frac{3}{2}(a_{\phi+\lambda h}^{-1}(s)-\phi(x)-\lambda h(x))_+^{1/2}h(x)\\
&&\hspace*{-1cm}+\frac{3}{2}(a_{\phi+\lambda h}^{-1}(s)-\phi(x)-\lambda h(x))_+^{1/2}\frac{\ds \int_{\RR^3}(a_{\phi+\lambda h}^{-1}(s)-\phi(x)-\lambda h(x))_+^{1/2}h(x)\,dx}{\ds \int_{\RR^3}(a_{\phi+\lambda h}^{-1}(s)-\phi(x)-\lambda h(x))_+^{1/2}\,dx}
\eee
with the uniform estimate: for all $s\in [0,L_0]$, $\lambda\in [0,1]$, $x\in \RR^3$
\be
\label{estderuniforme}
\left|\frac{\pa}{\pa \lambda}g(\lambda,x,s)\right|\leq 3(e_2+|\min\phi|+|\min\widetilde \phi|)_+^{1/2}\|h\|_{L^\infty}.
\ee
Since the function $s\mapsto Q^*(s)$ is monotone decreasing from $\|Q\|_{L^\infty}$ to 0, the function ${Q^*}'$ belongs to $L^1(0,L_0)$, and hence the uniform  domination \fref{estderuniforme} allows us to apply Lebesgue's derivation theorem and get:
\bea
\label{derivseconde}
\nonumber & &  \frac{\pa^2}{\pa\lambda^2}\mathcal J_0(\phi+\lambda h)= 4\pi\sqrt{2}\int_0^{L_0}ds\,{Q^*}'(s)\int_{\RR^3}(a_{\phi+\lambda h}^{-1}(s)-\phi(x)-\lambda h(x))_+^{1/2}(h(x))^2dx\\
& & \qquad  -4\pi\sqrt{2}\int_0^{L_0}ds\,{Q^*}'(s)\frac{\left(\ds \int_{\RR^3}(a_{\phi+\lambda h}^{-1}(s)-\phi(x)-\lambda h(x))_+^{1/2}h(x)\,dx\right)^2}{\ds \int_{\RR^3}(a_{\phi+\lambda h}^{-1}(s)-\phi(x)-\lambda h(x))_+^{1/2}\,dx}\,.
\eea

\bs
\ni
{\em Step 3. Identification of the first and second derivative of $\mathcal J$ at $\phi_Q$.}

\ms
\ni
Let $\phi\in \X$ and $h=\phi-\phi_Q$. We claim that
\be
\label{first}
D\mathcal J(\phi_Q)(h)=0.
\ee
Indeed, first remark from \fref{JJ} that
\be
\label{first2}
D\mathcal J(\phi_Q)(h)=D\mathcal J_0(\phi_Q)(h)+\int_{\RR^3}\na\phi_Q\cdot \na h\,dx.
\ee
Next, by \fref{lam1} and \fref{Qfixe}:
\bee
D\mathcal J_0(\phi_Q)(h) &=& 4\pi\sqrt{2}\int_{-\infty}^0\int_{\RR^3}Q^*\left(a_{\phi_Q}(e)\right)(e-\phi_Q(x))_+^{1/2}h(x)\,dxde\\
&=& 4\pi\sqrt{2}\int_{-\infty}^0\int_{\RR^3}F(e)(e-\phi_Q(x))_+^{1/2}h(x)\,dxde.
\eee
Applying the change of variable $e\mapsto u=\sqrt{2(e-\phi_Q(x))}$, it comes
\bee
D\mathcal J_0(\phi_Q)(h)&=&
4\pi\int_{0}^{+\infty}\int_{\RR^3}F\left(\frac{u^2}{2}+\phi_Q(x)\right)\,h(x)u^2dudx\\
&=&\int_{\RR^6}Q(x,v)h(x)\,dxdv,
\eee
where we used the expression \fref{Q} of $Q$. Hence, from the Poisson equation, we deduce after an integration by parts that
$$D\mathcal J_0(\phi_Q)(h)=-\int_{\RR^3}\na\phi_Q\cdot \na h\,dx,
$$
which together with \fref{first2} implies \fref{first}.

Let us now identify the right second derivative of $\mathcal J$ at $\phi_Q$. We have
\be
\label{seconde}
D^2\mathcal J(\phi_Q)(h,h)=D^2\mathcal J_0(\phi_Q)(h,h)+\int_{\RR^3}|\na h|^2\,dx
\ee
and, by \fref{derivseconde},
\bee
& & D^2\mathcal J_0(\phi_Q)(h,h)= 4\pi\sqrt{2}\int_0^{L_0}ds\,{Q^*}'(s)\int_{\RR^3}(a_{\phi_Q}^{-1}(s)-\phi_Q(x))_+^{1/2}(h(x))^2dx\\
& & \qquad  -4\pi\sqrt{2}\int_0^{L_0}ds\,{Q^*}'(s)\frac{\left(\ds \int_{\RR^3}(a_{\phi_Q}^{-1}(s)-\phi_Q(x))_+^{1/2}h(x)\,dx\right)^2}{\ds \int_{\RR^3}(a_{\phi_Q}^{-1}(s)-\phi_Q(x))_+^{1/2}\,dx}\,.
\eee
Using first the change of variable $s\mapsto e=a_{\phi_Q}^{-1}(s)$, \fref{aphiprime} and $F=Q^*\circ a_{\phi_Q}$, we get
\bee
D^2\mathcal J_0(\phi_Q)(h,h)&=& 4\pi\sqrt{2}\int_{-\infty}^0de\,F'(e)\int_{\RR^3}(e-\phi_Q(x))_+^{1/2}(h(x))^2dx\\
&&-4\pi\sqrt{2}\int_{-\infty}^0\,F'(e)\frac{\left(\ds \int_{\RR^3}(e-\phi_Q(x))_+^{1/2}h(x)\,dx\right)^2}{\ds \int_{\RR^3}(e-\phi_Q(x))_+^{1/2}\,dx}\,.
\eee
We next apply the change of variable $e\mapsto u=\sqrt{2(e-\phi_Q(x))}$ to get
\bee
D^2\mathcal J_0(\phi_Q)(h,h) & = & \int_{\RR^6}F'(e)(h(x))^2\,dxdv-\int_{\RR^6}F'(e)h(x)\Pi h(e) dxdv\\
& = & -\int_{\RR^6}|F'(e)|(h(x)-\Pi h(e))^2 dxdv,
\eee
where we used the shorthand notation $e=\frac{|v|^2}{2}+\phi_Q(x)$ and the fact that $\Pi$ given by \fref{phiphi} is the projector onto the functions which depend only on $e$. This together with \fref{seconde} concludes the proof of \fref{deriv2}.

\bs
\ni
{\em Step 4. Proof of the Taylor expansion \fref{taylor}.}

\ms
\ni
Let $\phi\in \X$ and $h=\phi-\phi_Q$. We first deduce from the fact that $\mathcal J(\phi_Q+\lambda h)$ is twice differentiable with respect to $\lambda$ that
$$
\mathcal J(\phi_Q+h)-\mathcal J(\phi_Q)=\int_0^1 (1-\lambda)\frac{\pa^2}{\pa \lambda^2}\mathcal J(\phi_Q+\lambda h)\,d\lambda
$$
and hence:
\bea
\label{qsd}
& &  \mathcal J(\phi_Q+h)-\mathcal J(\phi_Q)-\frac{1}{2}D^2\mathcal J(\phi_Q)(h,h)\\
\nonumber &= & \int_0^1 (1-\lambda)\left(D^2\mathcal J(\phi_Q+\lambda h)-D^2\mathcal J(\phi_Q)\right)(h,h)\,d\lambda
\\
\nonumber & = & \|\na h\|_{L^2}^2\int_0^1 (1-\lambda)\left(D^2\mathcal J_0(\phi_Q+\lambda h)-D^2\mathcal J_0(\phi_Q)\right)\left(\frac{h}{\|\na h\|_{L^2}},\frac{h}{\|\na h\|_{L^2}}\right)\,d\lambda.
\eea
We now claim the following continuity property:
\be
\label{conti2}
\sup_{\lambda\in [0,1]}\sup_{\|\na \widetilde h\|_{L^2}=1}\left|\left(D^2\mathcal J_0(\phi_Q+\lambda (\phi-\phi_Q))-D^2\mathcal J_0(\phi_Q)\right)(\widetilde h,\widetilde h)\right|\to 0
\ee
as $\|\phi-\phi_Q\|_{L^\infty}\to 0$. Assume \fref{conti2}, then
$$\lim_{\|h\|_{L^\infty}\to 0}\int_0^1 (1-\lambda)\left(D^2\mathcal J_0(\phi_Q+\lambda h)-D^2\mathcal J_0(\phi_Q)\right)\left(\frac{h}{\|\na h\|_{L^2}},\frac{h}{\|\na h\|_{L^2}}\right)\,d\lambda=0\,$$
and  \fref{qsd} now yields the Taylor expansion \fref{taylor}.\\
{\em Proof of \fref{conti2}}. We argue by contradiction and assume that there exists $\eps>0$, $H_n$, $\widetilde h_n$ and $\lambda_n\in [0,1]$  such that 
\be
\label{contr1}
\|H_n\|_{L^\infty}\leq \frac{1}{n}\,,\quad \|\na \widetilde h_n\|_{L^2}=1\,,
\ee
and
\be
\label{contr2}
\left|D^2\mathcal J_0(\phi_Q+ \lambda_n H_n)(\widetilde h_n,\widetilde h_n)-D^2\mathcal J_0(\phi_Q)(\widetilde h_n,\widetilde h_n)\right|>\eps.\ee
We denote $h_n=\lambda_n H_n$. Recall from \fref{derivseconde}:
\be
\label{d1quabis}
\begin{array}{l}
\ds D^2\mathcal J_0(\phi_Q+h_n)(\widetilde h_n,\widetilde h_n)=\\[4mm]
\ds \qquad =4\pi\sqrt{2}\int_0^{L_0}ds\,{Q^*}'(s)\int(a_{\phi_Q+h_n}^{-1}(s)-(\phi_Q+h_n)(x))_+^{1/2}(\widetilde h_n(x))^2\,dx\\[4mm]
\ds \qquad \quad -4\pi\sqrt{2}\int_0^{L_0}ds\,{Q^*}'(s)\frac{\ds \left(\int(a_{\phi_Q+h_n}^{-1}(s)-(\phi_Q+h_n)(x))_+^{1/2}\widetilde h_n(x)\,dx\right)^2}{\ds \int(a_{\phi_Q+h_n}^{-1}(s)-(\phi_Q+h_n(x))_+^{1/2}\,dx}.
\end{array}
\ee
Let us analyze the sequence $e_n=a_{\phi_Q+h_n}^{-1}(s)$. We claim that:
\be
\label{cla}
\forall s\in(0,L_0), \ \ \lim_{n\to+\infty}e_n=a_{\phi_Q}^{-1}(s).
\ee
Indeed, we observe 
$$s=a_{\phi_Q+h_n}(e_{n})\geq  \frac{8\pi\sqrt{2}}{3}\int \left(e_{n}-\phi_Q(x)-\|h_n\|_{L^\infty}\right)_+^{3/2}dx\geq a_{\phi_Q}\left(e_{n}-\frac{1}{n}\right)$$
which yields $e_{n}\leq a_{\phi_Q}^{-1}(s) +1/n$. Similarly, we also have $e_{n}\geq a_{\phi_Q}^{-1}(s) -1/n$ and \fref{cla} follows.

Let us now pass to the limit in \fref{d1quabis}. Note first that the domain of integration in $x$ of these integrals is uniformly bounded as $n\to +\infty$. Indeed, the set of integration is 
$$D_n(e_{n}):=\left\{x\in\RR^3:\,\phi_Q(x)+h_n(x)<e_{n}\right\}\subset \left\{x\in \RR^3:\,\phi_Q(x)\leq e_{n}+1/n\right\},$$
which is bounded for $n$ large enough, since $e_{n}\leq a_{\phi_Q}^{-1}(L_{0}) +1/n\leq \frac{1}{2}a_{\phi_Q}^{-1}(L_{0})$, and the continuous function $\phi_Q$ converges to zero at infinity.

Now the local compactness of the Sobolev embedding $\dot{H}^1\hookrightarrow L^p_{loc}$ for $1\leq p<6$ implies that there exists $\widetilde h\in \dot H^1_{rad}$ such that --up to a subsequence-- 
\be
\label{clocasopr}
\widetilde h_n\to \widetilde h \ \ \mbox{in} \ \ L^2_{loc} \ \ \mbox{as} \ \ n\to +\infty.
\ee
 Hence, for all $s\in (0,L_0)$ and for $i=0,1,2$, \fref{cla}, \fref{clocasopr} ensure:
\bee
& &  \hspace*{-1cm}\int_{\RR^3}(a_{\phi_Q+h_n}^{-1}(s)-(\phi_Q+h_n)(x))_+^{1/2}(\widetilde h_n(x))^i\,dx\\
&&= \int_{|x|\leq R}(a_{\phi_Q+h_n}^{-1}(s)-(\phi_Q+h_n)(x))_+^{1/2}(\widetilde h_n(x))^i\,dx\\
& &\to \int_{\RR^3}(a_{\phi_Q}^{-1}(s)-\phi_Q(x))_+^{1/2}(\widetilde h(x))^i\,dx \ \ \mbox{as} \ \ n\to +\infty.
\eee
Moreover, by Cauchy-Schwarz and $a_{\phi_Q+h_n}^{-1}(s)\leq 0$:
\bee
 \frac{\ds \left(\int(a_{\phi_Q+h_n}^{-1}(s)-\phi_Q-h_n)_+^{1/2}\widetilde h_n\,dx\right)^2}{\ds \int(a_{\phi_Q+h_n}^{-1}(s)-\phi-h_n)_+^{1/2}\,dx}\leq \int(a_{\phi_Q+h_n}^{-1}(s)-\phi_Q-h_n)_+^{1/2}(\widetilde h_n)^2\,dx&&\\
 \lesssim  \int_{|x|\leq R} (\|\phi_Q\|_{L^\infty}+\|h_n\|_{L^\infty})^{1/2}(\widetilde h_n)^2dx\lesssim 1.&&
\eee
Recall now that the function ${Q^*}'$ is $L^1$ on $[0,L_0]$, since $Q^*$ is decreasing and bounded. Therefore, Lebesgue's convergence theorem applied to (\ref{d1quabis}) yields:
$$D^2\mathcal J_0(\phi_Q+h_n)(\widetilde h_n,\widetilde h_n)\to D^2\mathcal J_0(\phi_Q)(\widetilde h,\widetilde h).$$
A similar argument gives:
\be
\label{comp}
D^2\mathcal J_0(\phi_Q)(\widetilde h_n,\widetilde h_n)\to D^2\mathcal J_0(\phi_Q)(\widetilde h,\widetilde h)
\ee
as $n\to +\infty$. This contradicts \fref{contr2} and concludes the proof of \fref{conti2}.

The proof of Proposition \ref{lemA1} is complete.
\end{proof}

\begin{remark}
\label{compact}
We have proved in this last Step 4 that for all sequence $\widetilde h_n$ bounded in $\dot H^1$, after extraction of a subsequence, we have the strong convergence \fref{comp}. Hence the quadratic form $D^2\mathcal J_0(\phi_Q)$ is compact on $\dot H^1$.
\end{remark}


\subsection{A new Antonov type inequality}


We now turn to the second key of our analysis which is a generalization of the celebrated Antonov's stability property --see Proposition 4.1 in \cite{LMR8} for a precise statement--:

\begin{proposition}[Generalized Antonov's stability property]
\label{antonov}
Let $Q$ satisfy the assumptions of Theorem \pref{thm} and consider the linear operator generated by the Hessian \fref{deriv2}: $$\mathcal L h=-\Delta h-\int_{\RR^3}|F'(e)|(h-\Pi h)dv.$$ 
Then $\mathcal L$ is a compact perturbation of the Laplacian operator on $\dot H^1$ and is positive: 
\be
\label{kernelnonradial}
\forall h\in \dot H^1, \ \ (\mathcal Lh,h)=D^2\mathcal J(\phi_Q)(h,h)\geq 0.
\ee Moreover, $$\mbox{Ker}(\mathcal L)=\{h\in \dot H^1 \ \ \mbox{with} \ \ \mathcal Lh=0\}=\mbox{Span}(\pa_{x_i}\phi_Q)_{1\leq i\leq 3}.$$ In particular, there exists $c_0>0$ such that
\be
\label{coercivityradial}
\forall h\in \dot H^1, \ \ (\mathcal Lh,h)\geq c_0 \|\nabla h\|_{L^2}^2-\frac{1}{c_0}\sum_{i=1}^3\left(\int_{\RR^3} h\Delta (\pa_{x_i}\phi_Q)\right)^2.
\ee
\end{proposition}

\begin{remark} The fact that the kernel is completely explicit and purely generated by the symmetry group is remarkable and reminiscent from similar statements in dispersive equations, see Weinstein \cite{W1}, the coercivity on the radial sector being always the most delicate problem.
\end{remark} 

\bs
\ni
{\em Proof of Proposition \ref{antonov}}

\ms
\ni
{\em Step 1. Positivity away from radial modes.}

\ms
\ni
Let $h\in \dot H^1_{rad}$, and let us introduce the projection of $h$ onto the radial sector 
$$h_0(r)=\frac{1}{4\pi}\int_{\SS^2}h(r\sigma)d\sigma,$$
where $\SS^2$ denotes the unit sphere in $\RR^3$ and $d\sigma$ denotes the surface measure on $\SS^2$ induced by the Lebesgue measure. We have the decomposition
$$h=h_0+h_1,  \  \ h_0\in \dot H^1_{rad}, \ \ h_1\in (\dot H^1_{rad})^{\perp}.$$ The angular integration in \fref{phiphi} ensures $\Pi h_1=0$ and thus $$(\mathcal Lh,h)=(\mathcal Lh_0,h_0)+\int_{\RR^3} |\nabla h_1|^2-\int_{\RR^3}V_Q\,h_1^2$$
with 
$$V_Q(r)=\int_{\RR^3}|F'(e)|dv=4\pi\sqrt{2}\int_{\phi_Q(0)}^0 |F'(e)|\left(e-\phi_Q(r)\right)_+^{1/2}de,$$
where we applied the change of variable $e=\frac{|v|^2}{2}+\phi_Q(r)$. Since $F'(e)<0$ and $F$ is bounded on $[\phi_Q(0), 0]$, the function $|F'|$ belongs to $L^1$. Therefore by dominated convergence, the function $V_Q$ is continuous. Moreover, since $F(e)=0$ for $e\geq e_0$ and since $\phi_Q$ is strictly increasing, we have:
$$\mbox{Supp}(V_Q)=[0,(\phi_Q)^{-1}(e_0)].$$
Hence, $V_Q$ being continuous and compactly supported, the Schr\"odinger operator $-\Delta -V_Q$ is a compact perturbation of the Laplacian on $\dot H^1$. Observe that $\phi_Q'$ (and also $\pa_{x_i}\phi_Q$ for $i=1,\cdots,3$) belongs to $\dot H^1(\RR^3)$. Translating the $\phi_Q$ equation yields:
$$\Delta \phi_Q(x+x_0)=\rho_Q(x+x_0)=\frac{8\pi\sqrt{2}}{3}\int_{-\infty}^0|F'(e)|(e-\phi_Q(x+x_0))_+^{3/2}de$$
and differentiating this relation with respect to $x_0$ yields at $x_0=0$: 
\be
\label{ceofeiofu}
\mathcal L (\nabla \phi_Q)=0.
\ee 
We now claim from standard argument that this implies the positivity of $\mathcal L$ away from radial modes, see \cite{W1} for related statements:
\be
\label{cneofieofeu}
\forall h\in \left(\dot{H}_{rad}^1\right)^\perp,  \ \ (\mathcal Lh,h)\geq 0,
\ee
and 
\be
\label{cnopfjepoeuvnkldv}
\{h\in (\dot H_{rad}^1)^\perp \ \ \mbox{with} \ \ \mathcal Lh=0\}=\mbox{Span}(\pa_{x_i}\phi_Q)_{1\leq i\leq 3},
\ee
Let us briefly recall the argument. Let us decompose $h\in (\dot{H}^1_{rad})^{\perp}$ into spherical harmonics, $$h=\sum_{k\geq 1} \sum_jh_{k,j}Y_{k,j}(\hat{x})$$ where $\hat{x}=\frac{x}{r}$ is the spherical variable and $-\Delta_{\Bbb S^2}Y_{k,j}=\lambda_kY_{k,j}.$ Then the radiality of $V_Q$ ensures the orthogonal decomposition $$\mathcal L h=\sum_{k\geq 1}\sum_j A_kh_{k,j}$$ with 
\be
\label{cneoeifofejio}
A_k=-\pa_{r}^2-\frac{2}{r}\pa_r+\frac{\lambda_k}{r^2}-V_Q(r), \ \ \lambda_k=k(k+1).
\ee
 For $k=1$, we have $\nabla \phi_Q=\phi_Q'(r)\hat{x}$ and \fref{ceofeiofu} implies $A_1\phi_Q'=0$. Since $\phi_Q'>0$ for $r>0$, $\phi_Q'\in \dot{H}^1$, $\phi_Q'$ is from standard Sturm Liouville results the ground state of $A_1$ which is thus positive with kernel on $\dot{H}^1_{rad}$ spanned by $\phi_Q'$. Now \fref{cneoeifofejio} ensures that $A_k$ is definite positive on $\dot{H}^1_{rad}$ for $k\geq 2$ and \fref{cneofieofeu}, \fref{cnopfjepoeuvnkldv} follow.

\bs
\ni
{\em Step 2. Coercivity away from radial modes.}

\ms
\ni
We now claim:
\be
\label{coercivityradialbis}
\forall h\in (\dot{H}_{rad}^1)^\perp, \ \ (\mathcal Lh,h)\geq c_1 \|\nabla h\|_{L^2}^2-\frac{1}{c_1}\sum_{i=1}^3\left(\int_{\RR^3} h\Delta(\pa_{x_i}\phi_Q)\right)^2
\ee
for some universal constant $c_1>0$. Let us briefly recall the argument which is standard. From \fref{cneofieofeu}, $$I=\inf \left\{(\mathcal Lh,h), \ \ h\in (\dot{H}_{rad}^1)^\perp, \ \ \int_{\RR^3} V_Qh^2=1, \ \ \int_{\RR^3} h\Delta(\pa_{x_i}\phi_Q)=0\right\}\geq 0.$$ We argue by contradiction and assume $I=0$, then there exists a sequence $h_n\in  (\dot{H}_{rad}^1)^\perp$ with $$\int_{\RR^3} V_Qh_n^2=1, \ \ \int_{\RR^3}|\nabla h_n|^2-\int_{\RR^3}V_Qh_n^2\leq \frac1n, \ \ \int_{\RR^3}h_n\Delta(\pa_{x_i}\phi_Q)=0.$$ From Sobolev embeddings, $h_n\to h$ in $L^p_{loc}$, $1\leq p<6$, up to a subsequence. Moreover, from \fref{ceofeiofu}, $ \Delta (\pa_{x_i}\phi_Q)$ is compactly supported and in $L^2$ from which passing to the limit  yields 
\be
\label{cneokeofueioueo}
(\mathcal Lh,h)\leq 0, \ \  \int_{\RR^3}h\Delta(\pa_{x_i}\phi_Q)=0, \ \ \int_{\RR^3} V_Qh^2=1
\ee 
and hence $h\neq 0$ attains the infimum. From Lagrange multipier theory, we thus can find $(\lambda_i)_{1\leq i\leq 3} $ with $$\mathcal Lh=\lambda_0V_Qh+\sum_{i=1}^3\lambda_i \Delta(\pa_{x_i}\phi_Q).$$ Taking the inner product with $h$ yields $\lambda_0=0$, then with $\pa_{x_i}\phi_Q$ yields $\lambda_i=0$, and thus $\mathcal L h=0$. From \fref{cnopfjepoeuvnkldv}, $h\in  \mbox{Span}(\pa_{x_i}\phi_Q)_{1\leq i\leq 3}$, but this contradicts the orthogonality relations \fref{cneokeofueioueo}, and \fref{coercivityradialbis} follows.

\bs
\ni
{\em Step 3. Strategy: H\"ormander's proof of Poincar\'e inequality.}

\ms
\ni
The relative compactness of $\mathcal L$ with respect to $\Delta$ in $\dot H^1$ follows from Remark \ref{compact}. It thus remains to prove \fref{coercivityradial} which from \fref{coercivityradialbis} and the Fredholm alternative is equivalent to:
\be
\label{tobeprovedformequadra}
\forall h\in \dot H^1_{rad}, \ \ h\neq 0, \ \ (\mathcal L h,h)>0.
\ee
Our main observation is now from \fref{deriv2} that \fref{tobeprovedformequadra} is nothing but a {\it Poincar\'e inequality with sharp constant}, and we now claim that we can adapt the celebrated proof by H\"ormander \cite{horm1,horm2} to our setting. H\"ormander's approach involves two key steps: the introduction of a self-adjoint operator adapted to the projection involved, and a suitable convexity property. The operator will be given by 
\be
\label{deft}
Tf(e,r)=\frac{1}{r^2\sqrt{2(e-\phi_Q(r))}}\pa_r f
\ee
which essentially satisfies the requirement $$\Pi h=0 \ \  \mbox{implies} \ \ h\in \mbox{Im}(T),$$ and the convexity will correspond to the lower bound: 
\be
\label{lowerboundconvexity}
-\frac{T^2g}{g}\geq \frac{3}{(r\sqrt{2(e-\phi_Q(r))})^4}\left(\rho_Q(r)+\frac{\phi_Q'(r)}{r}\right) 
\ee
with $$g(r,e)=\left(r\sqrt{2(e-\phi_Q(r))}\right)^3.$$
Note that the original proof of Antonov's stability criterion can be revisited as well using the transport operator $\tau=v\cdot\nabla_x-\nabla_{x} \phi_Q\cdot\nabla_v$ in the radial case as differential operator and whose image can be realized in the radial setting as the kernel of the {\it full} projection including the kinetic momentum $\ell$, see \cite{GR4}, \cite{LMR8} for more details.

\bs
\ni
{\em Step 4. Integration by parts.}

\ms
\ni
Recalling that $\phi_Q(r)$ is strictly increasing, for all $e\in (\phi_Q(0),0)$, we shall denote
$$r(e)=\phi_Q^{-1}(e).$$
Let $h\in\dot H^1_{rad}$ non zero.  Let \bee
\mathcal U&=&\left\{(r,e):\quad r>0, \ \ e\in(\phi_Q(0),0), \ \ e-\phi_Q(r)> 0\right\}\\
&=&\left\{(r,e):\quad e\in(\phi_Q(0),0), \ \ r\in(0,r(e))\right\}.
\eee
Given $\e>0$, we let $0\leq \chi_{\e}(e)\leq 1$ be a smooth cut off function such that 
$$\mbox{Supp}(\chi_\e)\subset (\phi_Q(0)+\e,e_0-\e), \ \ \chi_\e\equiv 1 \ \ \mbox{on} \ \ [\phi_Q(0)+2\e,e_0-2\e].$$ 
We let $$r_\e=r(\phi_Q(0)+\e).$$ 
Observe that
$$\int_0^{r(e)}\sqrt{e-\phi_Q(\tau)}\,\tau^2d\tau\geq c_{\e}>0 \ \ \mbox{on} \ \ \mbox{Supp}(\chi_\e),$$  and hence the radial interpolation estimate 
\be
\label{radial}
\|\sqrt{r}h(r)\|_{L^{\infty}(\RR^3)}\lesssim \|\nabla h\|_{L^2(\RR^3)}
\ee
and \fref{phiphi} ensure:
\be
\label{linftyph}
|\Pi h(e)|\leq C_\e  \ \ \mbox{on} \ \ \mbox{Supp}(\chi_\e).
\ee
Let us then define on $ \widetilde{\mathcal U}=\mathcal U\cap (0,r(e_0))\times (\phi_Q(0),e_0)$: 
\be
\label{deff}
f(r,e)=\int_0^r (h(\tau)-\Pi h(e))\sqrt{2(e-\phi_Q(\tau))}\,\tau^2d\tau.
\ee
Then $f$ is $\mathcal C^1$ with respect to the variable $r>0$ with
\be
\label{calculfonda}
Tf=h-\Pi h
\ee 
on $\widetilde{\mathcal U}$, where $T$ is given by \fref{deft}. Moreover, \fref{radial} and \fref{linftyph} yield the bound at the origin: $\forall e\in \mbox{Supp}(\chi_\e)$,
\be
\label{boundorigin}
|f(r,e)|\leq C_{\e}r^{5/2},
\ee
and  from \fref{phiphi} we get
$$
f\left(r(e),e\right)=\sqrt{2}\int_0^{+\infty} (h(\tau)-\Pi h(e))(e-\phi_Q(\tau))_+^{1/2}\,\tau^2d\tau=0.
$$
Hence, near the boundary $r=r(e)$, we estimate using  \fref{linftyph}: $\forall r\geq r_\eps$, \ \ $\forall e\in \mbox{Supp}\chi_\e$,
\be
\label{boundaryre}
|f(r,e)| =  \left|\int_r^{r(e)}(h(\tau)-\Pi h(e))\sqrt{2(e-\phi_Q(\tau))}\tau^2d\tau\right|\leq C_{\e}(e-\phi_Q(r))^{3/2},
\ee
where we used $e-\phi_Q(r)\sim C(r(e)-r)$ deduced from $\phi_Q'(r)\gtrsim \phi_Q'(r_\eps)>0$. 

We now integrate by parts from \fref{calculfonda} using the cancellations at the boundary of $\mathcal U$ given by  \fref{boundorigin}, \fref{boundaryre} and the bounds \fref{linftyph}, \fref{radial}:
\bee
& & \int_{\RR^6}\chi_{\e}|F'(e)|(h-\Pi h)^2dxdv \\
&& \qquad= 16\pi^2\int_{\phi_Q(0)}^0 \chi_{\e}|F'(e)|de\int_0^{r(e)} (h-\Pi h)Tf\sqrt{2(e-\phi_Q(r)}\,r^2dr\\
&&\qquad =16\pi^2 \int_{\phi_Q(0)}^0\chi_{\e}|F'(e)|de \int_0^{r(e)} (h-\Pi h)\pa_r fdr\\
&&\qquad =   -16\pi^2 \int_{\widetilde{\mathcal U}}\chi_{\e}|F'(e)|f\pa_r h \,dedr.
\eee
We now use Cauchy-Schwarz together with the identity 
$$\rho_Q(r)=\frac{8\pi\sqrt{2}}{3}\int_{\phi_Q(0)}^0|F'(e)|(e-\phi_Q(r))_+^{3/2}de$$
to estimate:
\bea
\label{cneoeiyoe}
\nonumber & &  \int_{\RR^6}\chi_{\e}|F'(e)|(h-\Pi h)^2dxdv  \\
\nonumber&&\qquad \leq   (4\pi)^{3/2}\|\nabla h\|_{L^2(\RR^3)}\left(\int_0^{r(e_0)}\frac{dr}{r^2}\left(\int_{\phi_Q(r)}^{e_0}\chi_\e|F'(e)|fde\right)^2\right)^{1/2}\\
\nonumber & &\qquad \leq (4\pi)^{3/2} \|\nabla h\|_{L^2(\RR^3)}\left[\frac{3}{8\pi\sqrt 2}\int_0^{r(e_0)}\frac{\rho_Q(r)}{r^2}dr\int_{\phi_Q(r)}^{e_0}\chi_\e|F'(e)|\frac{f^2}{(e-\phi_Q(r))^{3/2}}de\right]^{1/2}\\
& &\qquad =  \|\nabla h\|_{L^2(\RR^3)}\left(3\int\chi_\e\rho_Q(r)\frac{f^2}{r^4(\sqrt{2(e-\phi_Q)})^4}|F'(e)|dxdv\right)^{1/2}.
\eea
We now claim the following Hardy type control:
 \be
\label{keyhardy}
3\int\chi_\e\left(\rho_Q+\frac{\phi_Q'}{r}\right)\frac{f^2}{r^4(\sqrt{2(e-\phi_Q)})^4}|F'(e)|dxdv\leq \int \chi_\e|F'(e)|(Tf)^2dxdv.
\ee
Assume \fref{keyhardy}, then \fref{calculfonda} and \fref{cneoeiyoe} yield: 
$$\int_{\RR^6}\chi_{\e}|F'(e)|(h-\Pi h)^2dxdv+3\int\chi_\e\frac{\phi_Q'}{r}\frac{f^2}{r^4(\sqrt{2(e-\phi_Q)})^4}|F'(e)|dxdv\leq \int_{\RR^3}|\nabla h|^2dx.$$ 
Letting $\e\to 0$ now yields $(\mathcal Lh,h)\geq 0$. Moreover $(\mathcal L h,h)=0$ implies $f=0$ in $\tilde{\mathcal U}$, thus $h(r)=\Pi h(e)$ on $\tilde{\mathcal U}$ and $(\mathcal L h,h)=\int |\nabla h|^2=0$ and thus $h$ is zero. This concludes the  proof of \fref{tobeprovedformequadra}.

\bs
\ni
{\em Step 5. Hardy type control.}

\ms
\ni
The Hardy control \fref{keyhardy} is a consequence of the convexity estimate \fref{lowerboundconvexity}. Indeed, let $g$ be a given smooth function in $\tilde{\mathcal U}$, let $f=qg$ and compute:
\bea
\label{estheuieb}
\nonumber (Tf)^2 & = & (gT q+qT g)^2=g^2(T q)^2+g(T g)T(q^2)+q^2(T g)^2\\
\nonumber & = & g^2(T q)^2+T(q^2g(T g))-q^2((T g)^2+gT^2g)+q^2(T g)^2\\
& \geq & T(q^2gT g)-q^2gT^2g=T(q^2gT g)-\frac{T^2g}{g}f^2.
\eea
We now look for $g$ such that 
\be
\label{keyestimateg}
-\frac{T^2g}{g}\geq \frac{3}{\left[r\sqrt{2(e-\phi_Q)}\right]^4}\left(\rho_Q+\frac{\phi_Q'}{r}\right).
\ee
Let $$u=\sqrt{2(e-\phi_Q)} \ \ \mbox{so that} \ \ Tg(r,u)=\frac{\pa_r g}{r^2u}-\frac{\phi_Q'}{r^2u^2}\pa_ug,$$ and thus:
\bee
T^2g & = & \frac{1}{r^4u^2}\left[\pa^2_{rr}g-\frac{2}{r}\pa_rg\right]-\frac{\rho_Q}{r^4u^3}\pa_ug+\frac{\phi_Q'}{r^4u^4}\left[\pa_rg+4\frac{u}{r}\pa_ug-2u\pa^2_{ru}g\right]\\
& + & \frac{(\phi_Q')^2}{r^4u^5}\left[u\pa^2_{uu}g-2\pa_ug\right],
\eee 
where we used the Poisson equation satisfied by $\phi_Q$. The choice $g=r^3u^3$ yields: 
$$-\frac{T^2g}{g}=\frac{3}{r^4u^4}\left(\rho_Q+\frac{\phi_Q'}{r}\right).$$
Injecting this into \fref{estheuieb} and integrating on $\tilde{\mathcal U}$ yields: 
\bee
\int \chi_\e|F'(e)|(Tf)^2dxdv & \geq & 3\int\chi_\e\left(\rho_Q+\frac{\phi_Q'}{r}\right)\frac{f^2}{r^4(\sqrt{2(e-\phi_Q)})^4}|F'(e)|dxdv\\
& + & \int \chi_\e|F'(e)|T\left(f^2\frac{T g}{g}\right)dxdv.
\eee
The bounds \fref{boundorigin}, \fref{boundaryre} now justify the integration by parts
$$\int \chi_\e|F'(e)|T\left(f^2\frac{T g}{g}\right)dxdv=16\pi^2\int_{\phi_Q(0)}^{e_0}\chi_\e|F'(e)|de\int_0^{r(e)}\partial_r\left(f^2\frac{T g}{g}\right)dr=0
$$
and \fref{keyhardy} follows. This concludes the proof of Proposition \ref{antonov}.
\qed


\subsection{Proof of Proposition \ref{propJ}}


We are now in position to conclude the proof of Proposition \ref{propJ} which is a classical consequence of modulation theory coupled with the coercivity estimate \fref{coercivityradial}.

\bs
\ni
{\em Step 1. Implicit function theorem}

\ms
\ni
Given $\alpha>0$, let $U_{\alpha}=\{\phi\in \dot{H}^1(\RR^3); \ \ \|\nabla \phi-\nabla \phi_Q\|_{L^2}< \alpha \}$, and for $\phi \in \dot{H}^1$, $z\in \RR^3$, define
\be
\label{eps1}
  \e_{z}(x)=\phi(x+z) - \phi_Q(x).
\ee
We claim that there exists $\overline \alpha>0$, a neighbourhood $V$ of the origin in $\RR^3$ and a unique $C^1$ map $U_{\overline \alpha}\to V$ such that if $\phi\in U_{\overline\alpha}$, there is a unique $z\in V$ such that $\e_{z}$ defined as in (\ref{eps1}) satisfies
\be
\label{ortho}
\forall 1\leq i\leq 3, \qquad \int_{\RR^3} \e_{z}\Delta(\pa_{x_i}\phi_Q)dx=0. 
\ee
Moreover, there exists a constant $C>0$ such that if $u\in U_{\overline\alpha}$,  then 
 \be
 \label{bceihyeofuyeo}
 |z|+\|\na \eps_z\|_{L^2}\le C\|\nabla \phi-\nabla \phi_Q\|_{L^2}.
 \ee 
 Indeed, we define the following functionals of $(\phi,z)$: $$
  \mathcal F_i(\phi,z)= \int_{\RR^3} \e_{z}\Delta(\pa_{x_i}\phi_Q)dx, \ \ 1\leq i\leq 3 $$ and obtain at the point $(\phi,z) = (\phi_Q,0)$,
$$\frac{\partial \mathcal F_i}{\partial z_j}=-\delta_{ij}\|\na \pa_{x_i}\phi_Q\|_{L^2}^2.$$
The Jacobian of the above functional is $-\Pi_{i=1}^3\|\na \pa_{x_i}\phi_Q\|_{L^2}^2<0$, hence the implicit function theorem ensures the existence of $\overline \alpha>0$, a neighbourhood $V$ of the origin in 
$\RR^3$ and a unique $C^1$ map $U_{\overline{\alpha}}\rightarrow V$
such  that (\ref{ortho}) holds.

\bs
\ni
{\em Step 2. Conclusion}

\ms
\ni
Let $\phi\in \calX$ with $$\inf_{z\in \RR^3}\left(\|\phi-\phi_Q(\cdot-z)\|_{L^{\infty}}+\|\nabla\phi-\nabla\phi_Q(\cdot-z)\|_{L^{2}}\right)<\delta_0$$ for some small enough $\delta_0>0$ to be chosen later. Then there exists $z_1$ such that  
\be
\label{estzuniniti}
\|\phi-\phi_Q(\cdot-z_1)\|_{L^{\infty}}+\|\nabla\phi-\nabla\phi_Q(\cdot-z_1)\|_{L^{2}}<2\delta_0.
\ee
For $\delta_0\leq \frac{\overline{\alpha}}{2}$ small enough, we may apply Step 1 to $\phi(x+z_1)$ and find $z_2\in \RR^3$, $\e\in \dot{H}^1$ satisfiying the orthogonality conditions \fref{ortho} and the smallness \fref{bceihyeofuyeo} such that $\phi(x+z_1)=(\phi_Q+\e)(x-z_2)$, or equivalently 
\be
\label{zphi}
\phi(x)=(\phi_Q+\e)(x-z_{\phi}), \ \ z_{\phi}=z_1+z_2.
\ee
In fact, for $\delta_0$ small enough, a shift $z_\phi$ satisfying \fref{zphi}, the orthogonality conditions \fref{ortho} and the smallness condition \fref{bceihyeofuyeo}, is unique. This is a simple consequence of the uniqueness of the pair $(z_2,\e_{z_2})$ in Step 1. The continuity of the map $\phi\to z_{\phi}$ from $(\dot{H}^1,\|\cdot\|_{\dot{H}^1})\to \RR^3$ then follows. Moreover, from \fref{bceihyeofuyeo}, \fref{estzuniniti}: 
\bee
\|\e\|_{L^{\infty}} & = & \|\phi(x+z_1+z_2)-\phi_Q(x)\|_{L^{\infty}}\\
& \leq & \|\phi(x+z_1+z_2)-\phi_Q(x+z_2)\|_{L^{\infty}}+\|\phi_Q(x+z_2)-\phi_Q(x)\|_{L^{\infty}}\\
& \leq & \|\phi(x+z_1)-\phi_Q(x)\|_{L^{\infty}}+C|z_2|\\
& \leq & \|\phi(x+z_1)-\phi_Q(x)\|_{L^{\infty}}+C\|\nabla\phi(x+z_1)-\nabla\phi_Q(x)\|_{L^{2}}\leq C\delta_0.
\eee
Provided $\delta_0$ small enough, we may now apply the Taylor expansion \fref{taylor} together with the coercivity \fref{coercivityradial} and the orthogonality conditions \fref{ortho}, and obtain from the translation invariance of $\mathcal J$:
\bee
{\mathcal J}(\phi)-{\mathcal J}(\phi_Q) & = & {\mathcal J}(\phi_Q+\e)-{\mathcal J}(\phi_Q)\geq c_0\|\nabla \e\|_{L^2}^2-\eta(\|\e\|_{L^{\infty}})\|\nabla \e\|_{L^2}^2\geq \frac{c_0}{2}\|\nabla \e\|_{L^2}^2\\
 & \geq & \frac{c_0}{2}\|\nabla \phi-\nabla\phi_Q(\cdot-z_{\phi})\|_{L^2}^2.
 \eee
 This concludes the proof of Proposition \ref{propJ}.


\section{Compactness of local minimizing sequences of  the Hamiltonian}


The aim of this section is to prove  the following compactness result which is the heart of the proof of Theorem \ref{thm}.
 \begin{proposition}[Compactness of local minimizing sequences]
\label{theo-compacite-f}
Let  $\delta_{0}>0$ be as in Proposition \pref{propJ}. Let $\phi\to z_{\phi}$ the continuous map from $(\dot{H}^1,\|\cdot\|_{\dot{H}^1})\rightarrow \RR^3$ build in Proposition \pref{propJ}. Let $f_n$ be a sequence of functions of $\calE$, bounded in $L^\infty$, such that 
\be
\label{condfn1}
\inf_{z\in \RR^3}\left(\|\phi_{f_n}-\phi_Q(\cdot-z)\|_{L^{\infty}}+\|\nabla\phi_{f_n}-\nabla\phi_Q(\cdot-z)\|_{L^{2}}\right)<\delta_0,
\ee
and
\be
\label{condfn2}
\limsup_{n\to +\infty}\calH(f_n)\leq  \calH(Q),\qquad f_n^* \to Q^*  \mbox { in }  L^1(\RR_+)\quad \mbox{as }n\to +\infty.
\ee
Then 
\be
\label{convfn}
\int(1+|v|^2)|f_n-Q(x-z_{\phi_{f_n}})|\to 0 \ \ \mbox{as} \ \ n\to+\infty.
\ee
\end{proposition}

\begin{proof}

\medskip

{\em Step 1}:  {\em Compactness of the potential}

\medskip
\ni
 We first claim the following quantitative lower bound which generalizes the monotonicity formula \fref{keymononicity}: let $f\in \mathcal E$ such that $\phi_f$ satisfies \fref{linfitysmallnes}, let $z_{\phi_f}$ given by Proposition \ref{propJ}, then 
\be
\label{controlH-fstar-J} 
 \mathcal H(f) - \mathcal H(Q) + \|\phi_{f} \|_{L^{\infty}} \|f^*-Q^*\|_{L^1} \geq c_0\|\nabla \phi_f-\nabla \phi_{Q}(\cdot-z_{\phi_f})\|_{L^2}^2.
\ee
Indeed,
\be
\label{bla}\calH(f) -\calH(Q) \geq \mathcal J_{f^*}(\phi_{f}) - \mathcal J(\phi_{Q})=  \mathcal J_{f^*}(\phi_{f}) - \mathcal J(\phi_{f})+ \mathcal J(\phi_{f})- \mathcal J(\phi_{Q}),
\ee
where we have used that $\calH(Q)= \mathcal J(\phi_{Q})$. Now, we recall that
$$\mathcal J_{f^*}(\phi)= \int_{\RR^6}\left(\frac{|v|^2}{2}+\phi\right) f^{*\phi}(x,v) dx dv + \frac{1}{2} \int_{\RR^3}|\nabla \phi|^2 dx.,$$  and deduce from the change of variables formula \fref{changeofvariable} that
$$\mathcal J_{f^*}(\phi_{f}) - \mathcal J(\phi_{f})= \int_{0}^{+\infty}   a_{\phi_{f}}^{-1}(s)\left(f^*(s) - Q^*(s)\right) ds.$$
Since $|a_{\phi_{f}}^{-1}(s) |\leq - \mbox{min}\ \phi_{f} = \|\phi_{f}\|_{L^\infty}$, we have
$$\mathcal J_{f^*}(\phi_{f}) - \mathcal J(\phi_{f})\geq  - \|\phi_{f}\|_{L^\infty} \|f^* - Q^*\|_{L^1}.$$
Inserting this estimate into \fref{bla} and using Proposition \fref{controlH-fstar-J} yields \fref{controlH-fstar-J} .

Let  us now consider a sequence $f_{n}\in\mathcal E$ satisfying the assumptions of  Proposition \ref{theo-compacite-f},  then \fref{controlH-fstar-J}  applied to $f_{n}$ ensures:
\be
\label{conv-phin}
\|\nabla \phi_{f_{n}}(.+z_{\phi_{f_n}})- \nabla \phi_{Q}\|_{L^2} \to 0, \qquad \mbox{as}\  n\to \infty.
\ee

\bigskip
\ni
{\em Step 2:}  {\em Strong convergence of $f_{n}$ to $Q$}

\medskip
\ni
To ease notations, we shall still denote
by  $f_{n}$ the translated function $f_{n}(.+z_{\phi_{f_n}},v)$. We then observe the identity:
\be
\label{decomp-duale}
\mathcal H(f_{n})-\mathcal H(Q) + \frac{1}{2}\|\nabla \phi_{f_{n}}- \nabla \phi_{Q}\|_{L^2}^2= \int_{\RR^6}\left(\frac{|v|^2}{2} + \phi_{Q}(x)\right) (f_{n}-Q)dxdv
\ee
which implies, from \fref{condfn2} and \fref{conv-phin}, that
\be
\label{conv1}
\int_{\RR^6}\left(\frac{|v|^2}{2} + \phi_{Q}(x)\right) (f_{n}-Q)dxdv  \to 0, \qquad  \mbox{as}\  n\to \infty.
\ee
Now, we observe from the change of variables \fref{changeofvariable} that
$$\int_{\RR^6}\left(\frac{|v|^2}{2} + \phi_{Q}(x)\right) (Q-f_{n}^{*\phi_{Q}})dxdv = \int_{0}^{+\infty}   a_{\phi_{Q}}^{-1}(s)\left(Q^*(s)- f_{n}^{*}(s)\right) ds.$$ 
Since  $|a_{\phi_{Q}}^{-1}(s)|\leq -\phi_{Q}(0)$ we get
$$\left|\int_{\RR^6}\left(\frac{|v|^2}{2} + \phi_{Q}(x)\right) (Q-f_{n}^{*\phi_{Q}})dxdv\right| \leq   |\phi_{Q}(0)|\|Q^*- f_{n}^{*}\|_{L^1},$$ 
which implies from \fref{condfn2} that
\be
\label{conv2}
\int_{\RR^6}\left(\frac{|v|^2}{2} + \phi_{Q}(x)\right) (Q-f_{n}^{*\phi_{Q}})dxdv  \to 0, \qquad  \mbox{as}\  n\to \infty.
\ee
Summing \fref{conv1} and \fref{conv2} yields
\be
\label{conv3}
T_{n}=\int_{\RR^6}\left(\frac{|v|^2}{2} + \phi_{Q}(x)\right) (f_{n}-f_{n}^{*\phi_{Q}})dxdv  \to 0, \qquad  \mbox{as}\  n\to \infty.
\ee
We now argue as in the proof of \fref{identity2}, and  write \fref{conv3} in the following equivalent form
 \be
\label{monotonie-1}
 \ds T_{n}=\int_{t=0}^{+\infty} dt  \left(\int_{S_{1}^n(t)}  \left(\frac{|v|^2}{2}+\phi_{Q}(x)\right)dxdv   - \int_{S_{2}^n(t)}  \left(\frac{|v|^2}{2}+\phi_{Q}(x)\right)dxdv\right) \ \ \  \to 0,
 \ee
 where
$$S^n_{1}(t)=\{(x,v)\in \RR^6,  f_n^{*\phi_{Q}}(x,v)\leq t<f_n(x,v)\}, $$
$$ S^n_{2}(t)=\{(x,v)\in \RR^6,   f_n(x,v)\leq t< f_n^{*\phi_{Q}}((x,v)\}.$$
{}From \fref{implicationtwo}, we have

$$  \frac{|v|^2}{2}+\phi_{Q}(x)\geq (f^*_{n}\circ a_{\phi_{Q}})^{-1}(t),  \qquad  \forall (x,v) \in S_{1}^n(t).$$
Thus
\be
\label{inegalite1}T_{n}\geq \int_{t=0}^{+\infty} dt  \left(\int_{S_{1}^n(t)}  (f^*_{n}\circ a_{\phi_{Q}})^{-1}(t)dxdv   - \int_{S_{2}^n(t)}  \left(\frac{|v|^2}{2}+\phi_{Q}(x)\right)dxdv\right).
\ee
As a consequence of the equimeasurability  of $f_n^{*\phi_{Q}}$ and $f_{n}$, we know that 
$$ \mbox{meas}(S^n_{1}(t))=  \mbox{meas}(S^n_{2}(t)),$$
and then \fref{inegalite1} gives:
\be
\label{inegalite2}T_{n}\geq \int_{t=0}^{+\infty} dt  \int_{S_{2}^n(t)}  \left[(f^*_{n}\circ a_{\phi_{Q}})^{-1}(t)   -   \left(\frac{|v|^2}{2}+\phi_{Q}(x)\right)\right]dxdv.
\ee
{}From \fref{implicationone}, we have
$$ (f^*_{n}\circ a_{\phi_{Q}})^{-1}(t)\geq \frac{|v|^2}{2}+\phi_{Q}(x),  \qquad  \forall (x,v) \in S_{2}^n(t)$$
Thus, from \fref{conv3} and \fref{inegalite2}, we get
\be
\label{keyaeconv1}
A_{n}=\left[(f^*_{n}\circ a_{\phi_{Q}})^{-1}(t)   -   \left(\frac{|v|^2}{2}+\phi_{Q}(x)\right)\right]\un_{S_{2}^n(t)} (x,v)\to 0
\ee
as $n\to + \infty$, for almost every $(t,x,v)\in \RR_+\times \RR^3\times\RR^3$ (up to a subsequence). We now claim that this implies
\be
\label{keyaeconv2}
B_{n}=\left[(Q^*\circ a_{\phi_{Q}})^{-1}(t)   -   \left(\frac{|v|^2}{2}+\phi_{Q}(x)\right)\right]\un_{\overline S_{2}^n(t)} (x,v)\to 0, 
\ee
as $n\to + \infty$, for almost every $(t,x,v)\in \RR_+\times \RR^3\times\RR^3$, where
 $$\overline S^n_{2}(t)=\{(x,v)\in \RR^6,  f_n(x,v)\leq t< Q(x,v)\}. $$
 To prove \fref{keyaeconv2}, we write 
 $$S_{2}^n= \left(S_{2}^n\backslash \overline S_{2}^n\right) \cup \left(S_{2}^n\cap  \overline S_{2}^n\right), \ \ \ \ \ \  \overline S_{2}^n= \left(\overline S_{2}^n\backslash S_{2}^n\right) \cup \left(S_{2}^n\cap  \overline S_{2}^n\right),$$
 and get
 \be
 \begin{array}{lll}
 \label{decomposition}
 A_{n}-B_{n} &= &\ds\left(\frac{|v|^2}{2}+\phi_{Q}(x) - (Q^*\circ a_{\phi_{Q}})^{-1}(t) \right)\un_{\overline S_{2}^n(t)\backslash S_{2}^n(t)} \\
  & + & \ds  \left((f^*_{n}\circ a_{\phi_{Q}})^{-1}(t)-\frac{|v|^2}{2}-\phi_{Q}(x) \right) \un_{S_{2}^n(t)\backslash \overline S_{2}^n(t)} \\
 &+&\ds  \left[ (f^*_{n}\circ a_{\phi_{Q}})^{-1}(t)- (Q^*\circ a_{\phi_{Q}})^{-1}(t) \right]\un_{S_{2}^n(t)\cap  \overline S_{2}^n(t)}\,.
 \end{array}
 \ee
 We shall now examine the behavior of each of  these terms  when $n\to \infty$.
 We first observe that for all $g,h \in L^1(\RR^6)$ with $g\geq 0$, $h\geq 0$, we have
\be
\label{magicidentity}
\int_{0}^{+\infty}  \mbox{meas}\left\{g\leq t<h\right\} dt = \int_{\RR^6}(h-g)_+ dxdv \leq \| g- h \|_{L^1},
\ee
and thus from \fref{condfn2}:
 $$\int_0^{+\infty} \mbox{meas}(S_{2}^n(t)\backslash \overline S_{2}^n(t)) dt \leq \|f_{n}^{*\phi_{Q}}-Q\|_{L^1} = \|f_{n}^*-Q^*\|_{L^1} \ \ \to 0.$$
 Using in addition the estimate
 $$\left|(f_{n}^*\circ a_{\phi_{Q}})^{(-1)}(t)\right| \leq |\phi_Q(0)|,$$ 
 we deduce that the first two terms of the decomposition \fref{decomposition} go to $0$ almost everywhere when $n$ goes to infinity.
 We now treat the third term and show that,  for all $(t,x,v)$,
 \be
 \label{liminf} \liminf_{n\to \infty} \left[ (f^*_{n}\circ a_{\phi_{Q}})^{-1}(t)- (Q^*\circ a_{\phi_{Q}})^{-1}(t) \right]\un_{S_{2}^n(t)\cap  \overline S_{2}^n(t)} \geq 0.
 \ee
 To prove \fref{liminf}, we first use the strong $L^1$ convergence \fref{condfn2} to get
 \be
 \label{vohoehohve1}
\forall e\in(\phi_Q(0),0)\backslash A, \ \ f_n^*(a_{\phi_Q}(e))\to Q^*(a_{\phi_Q}(e)),
 \ee
 where $A$ is a zero-measure set in $\RR$, and  claim that
the monotonicity of $f_n^*$ in $e$ and the {\it continuity} of $Q^*$ in $e$ ensure:
 \be
 \label{vohoehohve}
\forall e\in(\phi_Q(0),0), \ \ f_n^*(a_{\phi_Q}(e))\to Q^*(a_{\phi_Q}(e)).
 \ee
Indeed, let $e\in (\phi_Q(0),0)$, and $(x_{p},y_{p} )\in (\phi_Q(0),0)\backslash A$   such that $x_{p}\leq e \leq y_{p}$ and $x_{p}\to e$, $y_{p}\to e$.  As $f_n^*\circ a_{\phi_{Q}}$ is decreasing, we have
 $$f_n^*\circ a_{\phi_{Q}}(y_{p}) \leq f_n^*\circ a_{\phi_{Q}}(e) \leq f_n^*\circ a_{\phi_{Q}}(x_{p}).$$
From \fref{vohoehohve1}  we then get
 $$Q^*(a_{\phi_Q}(y_{p}))\leq \liminf_{n\to \infty} f_n^*\circ a_{\phi_{Q}}(e)\leq  \limsup_{n\to \infty} f_n^*\circ a_{\phi_{Q}}(e)\leq Q^*(a_{\phi_Q}(x_{p})).$$
 Now we pass to the limit $p\to \infty$ and use the continuity of $Q^*\circ a_{\phi_Q}$ to get the claim \fref{vohoehohve}.
 
  Now, we turn back to the proof of \fref{liminf} and fix  $(t,x,v)$. Take then any $e$ such that 
 \be
 \label{econ}\phi_{Q}(0)<e<0, \ \ \ \mbox{and}\ \ \  Q^*(a_{\phi_{Q}}(e)) >t,
 \ee
 which implies from  \fref{vohoehohve}:
 $$f_{n}^*(a_{\phi_{Q}}(e)) > t,$$
 for $n$ large enough. Using the definition of the pseudo inverse given in Lemma \ref{lemmapseudoinverse}, we then obtain 
 $e\leq (f^*_{n}\circ a_{\phi_{Q}})^{-1}(t)$ for $n$ large enough, and hence
 $$e\leq \liminf_{n\to \infty}  (f^*_{n}\circ a_{\phi_{Q}})^{-1}(t).$$
 Since this equality  holds for all $e$ satisfying \fref{econ},  we conclude from the definition of the pseudo inverse $(Q^*\circ a_{\phi_{Q}})^{-1}(t)$ that
 $$\liminf_{n\to \infty}  (f^*_{n}\circ a_{\phi_{Q}})^{-1}(t)\geq (Q^*\circ a_{\phi_{Q}})^{-1}(t),$$
which yields \fref{liminf}.
 
 We now turn to  the decomposition \fref{decomposition} and get from \fref{liminf}
 $$\liminf (A_{n}-B_{n})\geq 0,\quad \mbox{for almost all }(t,x,v).$$
 Finally, observing that $B_{n}\geq 0$ and using \fref{keyaeconv1},  we conclude that
 \fref{keyaeconv2} holds true. Observe now that 
$$t<Q(x,v) \ \ \mbox{implies} \ \ Q(x,v)=F\left(\frac{|v|^2}{2}+\phi_Q(x)\right)>t.$$
By the assumptions of Theorem \ref{thm},  $e\to F(e)$ is continuous and strictly decreasing with respect to $e=\frac{|v|^2}{2}+\phi_Q(x)$ for $(x,v)\in \{Q>0\}$, and thus:
$$t<Q(x,v)\ \ \mbox{implies} \ \ (Q^*\circ a_{\phi_{Q}})^{(-1)}(t)- \frac{|v|^2}{2}-\phi_Q(x) >0.$$
We then deduce from \fref{keyaeconv2} and from $ \overline S^n_{2}(t)=\{(x,v): f_n(x,v)\leq t< Q(x,v)\} $ that, up to a subsequence extraction,
$$\un_{\{f_n\leq t<Q\}} \  \ \to  0, \ \ \mbox{as} \ \ n\to \infty,$$
for almost every $(t,x,v)\in\RR^*_+\times\RR^6$. 
Now from $\un_{\{f_n\leq t<Q\}} \leq \un_{\{t<Q\}} $ and
$$\int_0^{\infty} \int_{\RR^6} \un_{\{t<Q\}} dx dv dt = \|Q\|_{L^1} <+\infty.$$
we may apply the dominated convergence theorem to conclude:
$$\int_0^{\infty} \int_{\RR^6} \un_{\{f_n\leq t<Q\}} dx dv dt \to  0 \ \ \mbox{as} \ \ n\to \infty .$$
Injecting this into \fref{magicidentity} yields
\be
\label{Qfnplus}
\int_{\RR^6}(Q-f_{n})_{+} dx dv \to  0 \ \ \mbox{as} \ \ n\to \infty .
\ee 
Now we claim that, using $f_{n}^* \to Q^*$ in $L^1$, this implies
\be
\label{fnQplus}
\int_{\RR^6}(f_{n}-Q)_{+} dx dv \to  0 \ \ \mbox{as} \ \ n\to \infty .
\ee 
Indeed, we write
\bee
& & \hspace*{-1cm}\int_{\RR^6}(f_n-Q)_+ dxdv \leq  \int_{\RR^6}(f_n-f_n^{*\phi_{Q}})_+ dxdv +\int_{\RR^6}(f_n^{*\phi_{Q}}-Q)_+ dxdv \\
&&\leq   \int_{0}^{+\infty}\mbox{meas}\left\{f_n^{*\phi_{Q}}\leq t<f_n\right\} dt+\|f_n^{*\phi_{Q}}- Q \|_{L^1}\\
& &=  \int_{0}^{+\infty}\mbox{meas}\left\{f_n\leq t<f_n^{*\phi_{Q}}\right\} dt+\|f_n^{*}- Q^*\|_{L^1}  \\
&&=  \int_{\RR^6}(f_n^{*\phi_{Q}}-f_n)_+ dxdv+\|f_n^{*}- Q^* \|_{L^1}  \\
&&\leq  \int_{\RR^6}(Q-f_n)_+ dxdv+ \int_{\RR^6}(f_n^{*\phi_{Q}}-Q)_+dxdv +\|f_n^{*}- Q^* \|_{L^1}  \\
&&\leq  \int_{\RR^6}(Q-f_n)_+ dxdv+ 2\|f_n^{*}- Q^*\|_{L^1}  \\
\eee
where we repeatedly used \fref{magicidentity} and the fact that $f_n^{*\phi_{Q}} \in \Eq(f_n)$ implies
$$\forall t>0,\ \  \mbox{meas}\left\{f_n^{*\phi_{Q}}\leq t<f_n\right\}=\mbox{meas}\left\{f_n\leq t<f_n^{*\phi_{Q}}\right\}.$$
As $f_{n}^* \to Q^*$ in $L^1$, we then conclude that (\ref{Qfnplus}) implies (\ref{fnQplus}). 
Finally adding (\ref{Qfnplus}) and  (\ref{fnQplus}) gives
$$\| f_n- Q \|_{L^1} \to 0\ \mbox{as} \ \ n\to +\infty.$$
Furthermore,  \fref{condfn2} and the strong convergence $\na \phi_{f_n}\to \na \phi_Q$ in $L^2$ imply: $$\int_{\RR^6}|v|^2f_n\to \int_{\RR^6}|v|^2Q \ \ \mbox{ as }n\to +\infty,$$
Together with the a.e. convergence of $f_n$, this yields the strong $L^1$ convergence of $|v|^2f_n$ to $|v|^2Q$. Note that the uniqueness of the limit now implies the convergence of all the sequence $f_n$ which completes the proof of (\ref{convfn}).

This concludes the proof of Proposition \ref{theo-compacite-f}.

\end{proof}


\section{Non linear  stability of $Q$}


We now turn to the proof of the nonlinear stability result stated in Theorem \ref{thm}, which is a direct consequence of Proposition \ref{theo-compacite-f} and the known regularity of weak solutions to the Vlasov-Poisson system.

\bs
\ni
{\em Proof of Theorem \pref{thm}.}

\bigskip
\ni
{\em Step 1.}  {\em Continuity claim for weak solutions}

\ms
\ni
Let $f_0\in \calE$ and let $f(t)\in \calE$ be a corresponding weak solution to \fref{vp}. By the properties of weak solutions of the Vlasov-Poisson system \cite{dpl1,dpl2}, we have
\be
\label{cont2}
\forall t\geq 0,\qquad f(t)\in\Eq(f_0),\qquad \calH(f(t))\leq \calH(f_0).
\ee
We claim:
\be
\label{cont1}
\phi_f\in \calC([0,+\infty),L^\infty(\RR^3)\cap \dot{H}^1(\RR^3)).
\ee
Note that this implies from Proposition \ref{propJ} that 
\be
\label{cnofheoiyeo}
t\rightarrow z_{\phi_f(t)} \ \ \mbox{is continuous}.
\ee
To prove \fref{cont1}, recall that $f\in \calC([0,+\infty),L^1)$ (see \cite{dpl1,dpl2}) and hence \fref{cont1} follows from: $\forall f,g\in \mathcal E$,
\be
\label{se}
\|\nabla \phi_f-\nabla\phi_g\|_{L^2}+\|\phi_f-\phi_g\|_{L^\infty}\leq C_{f,g}\,\|f-g\|_{L^1}^{1/6},
\ee
where $C_{f,g}$ only depends on $\|f\|_\calE$ and $\|g\|_{\calE}$. Let us prove \fref{se}. First, from  H\"older:
$$
\forall x\in \RR^3,\quad |\phi_f-\phi_g|(x)=\left|\int_{\RR^3}\frac{\rho_f(y)-\rho_g(y)}{4\pi|x-y|}dy\right|\lesssim \|\rho_f-\rho_g\|_{L^{5/3}}^{5/6}\|\rho_f-\rho_g\|_{L^1}^{1/6},
$$
and from Hardy-Littlewood-Sobolev:
$$\|\nabla \phi_f-\nabla\phi_g\|_{L^2}\lesssim \|\rho_f-\rho_g\|_{L^{6/5}}\lesssim \|\rho_f-\rho_g\|_{L^1}^{7/12}\|\rho_f-\rho_g\|_{L^{5/3}}^{5/12}.$$
Second, by interpolation,
$$\|\rho_f-\rho_g\|_{L^{5/3}}\lesssim \|f-g\|_{L^\infty}^{2/5}\||v|^2(f-g)\|_{L^1}^{3/5}\leq C_{f,g}.$$
 Since 
$\|\rho_f-\rho_g\|_{L^1}\leq \|f-g\|_{L^1},$
this yields \fref{se} and the continuity \fref{cont1} of $\phi_f$ follows.

\bs
\ni
{\em Step 2:} {\em Conclusion.}

\medskip
\ni
An equivalent reformulation of Proposition \ref{theo-compacite-f} is the following:  for all $\e>0$ small enough, there exists $\eta>0$
such that if $f\in \calE$ with
\be
\label{e1}
\|f^*-Q^*\|_{L^1}\leq \eta,\quad \|f\|_{L^\infty}\leq \|Q\|_{L^\infty}+M,\quad \calH(f)\leq \calH(Q)+\eta
\ee
and
\be
\label{e2}
\inf_{z\in \RR^3}\left(\|\phi_{f}-\phi_Q(\cdot-z)\|_{L^{\infty}}+\|\nabla\phi_{f}-\nabla\phi_Q(\cdot-z)\|_{L^{2}}\right)<\delta_0,
\ee
then
\be
\label{e3}
\|(1+|v|^2)(f-Q(\cdot-z_{\phi_f}))\|_{L^1}\leq \e.
\ee

Let $\eps>0$ and let $\eta>0$ be the associated constant. We consider an initial data $f_0\in \mathcal E$ with
$$\|f_0-Q\|_{L^1}<\eta,\quad \|f_0\|_{L^\infty}\leq \|Q\|_{L^\infty}+ M \quad \mbox{and}\quad \calH(f_0)\leq \calH(Q)+ \eta$$ 
and a corresponding weak solution $f(t)$ of \fref{vp}. Observe that, by the contractivity of the symmetric rearrangement in $L^1$ (see \cite{lieb-loss}), we have
\be
\label{contract}
\|f_0^*-Q^*\|_{L^1}=\|f_0-Q\|_{L^1}\leq \eta.
\ee
Moreover,  \fref{se} implies that, for $\eta$ small enough, 
$$\|\nabla \phi_{f_0}-\nabla \phi_Q(\cdot-z_{\phi_{f_0}})\|_{L^2}+\|\phi_f(0)-\phi_Q(\cdot-z_{\phi_{f_0}})\|_{L^\infty}\leq \frac{\delta_0}{2}.$$ From \fref{cont2}, we first deduce that the corresponding solution $f(t)$ of \fref{vp} satisfies \fref{e1} for all $t\geq 0$. Hence, if we prove that 
\be
\label{w}
\forall t\geq 0,\qquad \|\nabla \phi_f(t)-\nabla \phi_Q(\cdot-z_{\phi_f(t)})\|_{L^2}+\|\phi_f(t)-\phi_Q(\cdot-z_{\phi_f(t)})\|_{L^\infty}<\delta_0,
\ee
then \fref{e3} holds true for all $t\geq 0$, which is nothing but \fref{eta2}. Now \fref{w} follows for $\eta>0$ small enough from a straightforward bootstrap argument using the continuity \fref{cont1}, \fref{cnofheoiyeo} and the bound \fref{se}. The proof of Theorem \pref{thm} is complete.
\qed


\begin{appendix}


\section{Proof of Lemma \ref{lemdiffa}}
\label{appendixa}

\begin{proof} The proof is similar to the one in \cite{LMR8} and we briefly sketch the argument for the sake of completeness. Recall that the set $\calX$ is convex, thus $\phi+\lambda h=(1-\lambda)\phi(x)+\lambda\widetilde \phi(x)$ belongs to $\calX$ for all $\lambda\in [0,1]$ and $a_{\phi+\lambda h}$ is well-defined.

\bs
\ni
{\em Step 1. Proof of (i).}

\ms
\ni
Let $e_1<0$ be fixed. For all $e\leq e_1$, we consider the domain
$$D_{\phi+\lambda h}(e)=\left\{x\in \RR^3:\quad (\phi+\lambda h)(x)<e\right\}.$$
{}From \fref{aphi}, we have
$$
a_{\phi+\lambda h}(e)=\frac{8\pi\sqrt{2}}{3} \int_{D_{\phi+\lambda h}(e)}\left(e-\phi(x)-\lambda h(x)\right)_+^{3/2}dx.
$$
We clearly have
$$D_{\phi+\lambda h}(e)\subset D_{\phi}(e_1)\cup D_{\widetilde\phi}(e_1).$$
Since $\phi(x)$ and $\widetilde \phi(x)$ go to zero at the infinity, $D_{\phi}(e_0)$ and $ D_{\widetilde\phi}(e_0)$ are bounded. Hence for all $e\leq e_1$, $D_{\phi+\lambda h}(e)$ is contained in a fixed compact domain of $\RR^3$. As in addition the functions $\phi$ and $\widetilde \phi$ are continuous, the Lebesgue dominated convergence theorem may thus be applied to obtain the continuity and the differentiability of $a_{\phi+\lambda h}(e)$ with respect to $\lambda$ and $e$. The expression \fref{derivaphi} follows.

\bs
\ni
{\em Step 2. Continuity of the function $\lambda\mapsto a_{\phi+\lambda h}^{-1}(s)$}.

\ms
\ni
Let $s\in \RR_+^*$. In this step, we prove that the function $\lambda\mapsto a_{\phi+\lambda h}^{-1}(s)$ is continuous. To this aim, we consider a sequence $\lambda_n\in [0,1]$ converging to $\lambda_0$ as $n\to +\infty$ and prove that $a_{\phi+\lambda_n h}^{-1}(s)$ converges to $a_{\phi+\lambda_0 h}^{-1}(s)$. We set
$$e_n=a_{\phi+\lambda_n h}^{-1}(s)\in (\min (\phi+\lambda_n h),0)\subset (\min \phi+\min \widetilde \phi,0).$$
Hence, up to a subsequence, $e_n$ converges to some $e\leq 0$ as $n\to +\infty$.

Let us prove that $e<0$ by contradiction. Assume that $e=0$. For $n$ large enough such that $\frac{\lambda_0}{2}\leq \lambda_n\leq \frac{1+\lambda_0}{2}$, we have
\bee
s=a_{\phi+\lambda_n h}(e_n)&=&\frac{8\pi\sqrt{2}}{3} \int_{\RR^3}\left(e_n-(1-\lambda_n)\phi(x)-\lambda_n \widetilde \phi(x)\right)_+^{3/2}dx\\
&\geq&\frac{8\pi\sqrt{2}}{3} \int_{\RR^3}\left(e_n-\frac{1-\lambda_0}{2}\phi(x)-\frac{\lambda_0}{2} \widetilde \phi(x)\right)_+^{3/2}dx=a_\psi(e_n),
\eee
where $\psi (x)=\frac{1-\lambda_0}{2}\phi(x)+\frac{\lambda_0}{2} \widetilde \phi(x)$. From Lemma \ref{lemaphi}, we have $\lim_{t\to 0-}a_\psi(t)=+\infty$, which implies that $\lim_{n\to+\infty}a_\psi(e_n)=+\infty$, a contradiction.

Therefore, we have $e_n\to e<0$. The continuity of $(\lambda,e)\mapsto a_{\phi+\lambda h}(e)$ proved in Step 1 gives that
$$s=a_{\phi+\lambda_n h}(e_n)\to a_{\phi+\lambda_0 h}(e)\mbox{ as }n\to +\infty.$$
Thus $e=a_{\phi+\lambda_0 h}^{-1}(s)$. This ends the proof of (ii).

\bs
\ni
{\em Step 3. Differentiability of $\lambda\mapsto a_{\phi+\lambda h}^{-1}(s)$}.

\ms
\ni
Denoting $\phi_0=\phi+\lambda_0h$ and $\phi_\lambda=\phi+\lambda h$, we write
\be
\label{aqw}
\begin{array}{ll}
\ds \frac{a^{-1}_{\phi_\lambda}(s)-a^{-1}_{\phi_0}(s)}{\lambda}
&\ds =\frac{a^{-1}_{\phi_\lambda}(s)-a^{-1}_{\phi_0}(s)}{a_{\phi_0}(a^{-1}_{\phi_\lambda}(s))-a_{\phi_0}(a^{-1}_{\phi_0}(s))}\,\frac{a_{\phi_0}(a^{-1}_{\phi_\lambda}(s))-a_{\phi_0}(a^{-1}_{\phi_0}(s))}{\lambda}\\[6mm]
&\ds=A_1(\lambda)\,A_2(\lambda),
\end{array}
\ee
where we have set
$$A_1(\lambda)=\frac{a^{-1}_{\phi_\lambda}(s)-a^{-1}_{\phi_0}(s)}{a_{\phi_0}(a^{-1}_{\phi_\lambda}(s))-a_{\phi_0}(a^{-1}_{\phi_0}(s))}\,,\quad A_2(\lambda)=\frac{a_{\phi_0}(a^{-1}_{\phi_\lambda}(s))-a_{\phi_\lambda}(a^{-1}_{\phi_\lambda}(s))}{\lambda},$$
and where we simply used that $a_{\phi_0}(a^{-1}_{\phi_0}(s))=s=a_{\phi_\lambda}(a^{-1}_{\phi_\lambda}(s))$. Let us examinate separately the convergence of the two factors $A_1$ and $A_2$ in \fref{aqw}. From Step 2, we have
$$\lim_{\lambda \to 0}a^{-1}_{\phi_\lambda}(s)=a^{-1}_{\phi_0}(s),$$
hence
\be
\label{aqw1}
\lim_{\lambda \to 0}A_1(\lambda)=\frac{1}{a_{\phi_0}'(a_{\phi_0}^{-1}(s))}=\frac{1}{\ds 4\pi\sqrt{2}\int_{\RR^3}(a^{-1}_{\phi_0}(s)-\phi_0(x))_+^{1/2}dx}.
\ee
Now, \fref{derivaphi} and Step 2 imply: \be
\label{aqw2}
\lim_{\lambda \to 0}A_2(\lambda)=
4\pi\sqrt{2}\int_{\RR^3}(a^{-1}_{\phi_0}(s)-\phi_0(x))_+^{1/2}h(x)dx.
\ee
Therefore, \fref{aqw}, \fref{aqw1} and \fref{aqw2} give \fref{derivaphim1}. This concludes the proof of Lemma \ref{lemdiffa}.
\end{proof}

\end{appendix}

\end{document}